\documentclass[11pt, a4paper]{amsart}

\usepackage{amsmath, amsthm, amssymb,fullpage}

\newcommand{\rd}{\mathrm{d}}
\newcommand{\bD}{\mathbb{D}}
\newcommand{\bC}{\mathbb{C}}
\newcommand{\C}{\mathbb{C}}
\newcommand{\bN}{\mathbb{N}}
\newcommand{\bZ}{\mathbb{Z}}
\newcommand{\bP}{\mathbb{P}}
\renewcommand{\P}{\mathbb{P}}
\newcommand{\bR}{\mathbb{R}}
\newcommand{\bQ}{\mathbb{Q}}
\newcommand{\sP}{\mathsf{P}}
\newcommand{\cS}{\mathcal{S}}
\newcommand{\sK}{\mathsf{K}}
\newcommand{\sA}{\mathsf{A}}
\newcommand{\sF}{\mathsf{F}}
\newcommand{\sJ}{\mathsf{J}}
\newcommand{\cZ}{\mathcal{Z}}
\newcommand{\diag}{\operatorname{diag}}
\newcommand{\cO}{\mathcal{O}}
\newcommand{\can}{\operatorname{can}}
\newcommand{\sI}{\mathsf{I}}
\newcommand{\supp}{\operatorname{supp}}

\newcommand{\Vv}{\boldsymbol{v}}

\theoremstyle{plain}
\newtheorem{theorem}{Theorem}[section]
\newtheorem{lemma}[theorem]{Lemma}
  
\newtheorem{mainth}{Theorem}
\theoremstyle{definition}

\newtheorem*{acknowledgement}{Acknowledgement}
\theoremstyle{remark}
\newtheorem{remark}[theorem]{Remark}

\numberwithin{equation}{section}

\begin{document} 
\title[]{Value distribution of derivatives in polynomial dynamics}

\author{Y\^usuke Okuyama}
\address{Division of Mathematics, Kyoto Institute of Technology, Sakyo-ku, Kyoto 606-8585 JAPAN}
\email{okuyama@kit.ac.jp}
\author{Gabriel Vigny}
\address{LAMFA, UPJV, 33 rue Saint-Leu, 80039 AMIENS Cedex 1, FRANCE}
\email{gabriel.vigny@u-picardie.fr}

\date{\today}

\subjclass[2010]{Primary 37F10: Secondary 37P30, 32H50}

\keywords{value distribution, higher derivative, iterated polynomials,
complex dynamics, non-archimedean dynamics, arithmetic dynamics, 
H\'enon-type polynomial automorphism of $\bC^2$}

\begin{abstract}
For every $m\in\bN$, 
we establish the equidistribution of the sequence of the averaged pull-backs
of a Dirac measure at any given
value in $\bC\setminus\{0\}$ under the $m$-th order
derivatives of the iterates of a polynomials $f\in\bC[z]$ of degree $d>1$ 
towards the harmonic measure of the filled-in Julia set of $f$
with pole at $\infty$. We also establish 
non-archimedean and arithmetic counterparts
using the potential theory on the Berkovich projective line
and the adelic equidistribution theory over a number field $k$
for a sequence of effective divisors on $\bP^1(\overline{k})$
having small diagonals and small heights. 

We show a similar result
on the equidistribution of the analytic sets where 
the derivative of each iterate of a H\'enon-type polynomial automorphism
of $\bC^2$ has a given eigenvalue. 
\end{abstract}

\maketitle 

\section{Introduction}
Let $f\in\bC[z]$ be a polynomial of degree $d>1$. The {\em filled-in Julia set} 
\begin{gather*}
K(f):=\Bigl\{z\in\bC:\limsup_{n\to\infty}|f^n(z)|<\infty\Bigr\} 
\end{gather*}
of $f$ is a {\em non-polar} compact subset in $\bC$. Let $g_f$
be the {\em Green function} of $K(f)$ with pole at $\infty$, 
regarding $\bP^1$ as $\bC\cup\{\infty\}$ (see e.g.\ \cite[\S4.4]{ransford}).
We extend $g_f$ as $=0$ on $K(f)$. For every $n\in\bN$, 
the difference $g_f-(\log\max\{1,|f^n|\})/d^n$ on $\bC$
is harmonic and bounded
near $\infty$ so it admits a harmonic
extension across $\infty$, and we have the estimate
\begin{gather}
g_f-\frac{\log\max\{1,|f^n|\}}{d^n}=O(d^{-n})\quad\text{as }n\to\infty\label{eq:Greeniterate}
\end{gather}
on $\bP^1$ uniformly.

Let us denote by $\delta_a$ the Dirac measure on $\bP^1$ at each $a\in\bP^1$.
The {\em harmonic measure} of $K(f)$ with pole at $\infty$
is the probability measure
\begin{gather*}
\mu_f:=\Delta g_f +\delta_\infty\quad\text{on }\bP^1,
\end{gather*}
which has no atoms on $\bP^1$ and is supported by $\partial K(f)$.
The exceptional set of $f$ is defined as
\begin{gather*}
 E(f):=\{a\in\bP^1:\#\bigcup_{n\in\bN\cup\{0\}}f^{-n}(a)<\infty\},
\end{gather*}
which 
consists of $\infty$ ($f^{-1}(\infty)=\{\infty\}$)
and at most one point $b\in\bC$ ($f^{-1}(b)=\{b\}$).
For every $h\in\bC(z)$ of $\deg h>0$ and every $a\in\bP^1$,
by the definition of the pullback operator $h^*$, we have 
$h^*\delta_a=\sum_{w\in h^{-1}(a)}(\deg_w h)\delta_a$ on $\bP^1$,
where $\deg_w h$ is the local degree of $h$ at $w$. 

Brolin \cite{brolin} studied the value distribution
of the iteration sequence $(f^n:\bP^1\to\bP^1)$ of $f$
and established that
{\em for every $a\in\bC\setminus E(f)$,}
\begin{gather*}
\lim_{n\to\infty}\frac{(f^n)^*\delta_a}{d^n}=\mu_f\quad\text{weakly on } \bP^1.
\end{gather*}
This equidistribution of pullbacks of points under iterations
initiated the study of value distribution of complex dynamics 
(see e.g.\ \cite[\S 6.5]{ransford}, \cite[\S VIII]{rudiments}, 
\cite{dinhsibony2,RS97}). 
In \cite[\S 2]{Gauthier_Vigny_derivative} and \cite[Theorem 1]{OkuyamaDerivatives},
a similar equidistribution statement 
replacing $f^n$ with the first order derivative $(f^n)'$ of $f^n$ 
has been proved first for $a\in\bC$ outside a polar set and 
then for any $a\in \C^*$, respectively. 

Our aim is to contribute to the study of the parallelism between the 
value distribution of the sequence of higher derivatives (or jets) of 
the iterations of $f$ and the value distribution of 
higher derivatives (or jets) of meromorphic mappings
(cf.\ \cite{Yamanoi13}), extending the results mentioned above
to several different settings;
higher derivatives of polynomials over various valued fields and  
H\'enon-type polynomial automorphisms of $\bC^2$.

\subsection{Over the field $\bC$ of complex numbers}
Let $f\in\bC[z]$ be a polynomial of degree $d>1$. 
For every $h\in\bC[z]$ and every $m\in\bN$, we write 
the $m$-th order derivative $\frac{\rd^m}{\rd z^m}h(z)$ of $h$ as $h^{(m)}$.

Our first principal result is the following. 

\begin{mainth}\label{th:anyorder}
Let $f\in\bC[z]$ be a polynomial of degree $d>1$, 
and $m\in\bN$. Then for every $a\in\bC\setminus\{0\}$,
\begin{gather}
 \lim_{n\to\infty}\frac{\bigl((f^n)^{(m)}\bigr)^*\delta_a}{d^n-m}=\mu_f\quad\text{weakly on }\bP^1.\label{eq:equidistderivative}
\end{gather}
\end{mainth}
In Theorem \ref{th:anyorder}, the values $a=0,\infty$ 
need to be excluded as for every 
$n\in\bN$, $((f^n)^{(m)})^*\delta_\infty/(d^n-m)=\delta_\infty\neq\mu_f$ and, if $b\in E(f)\cap\bC$, then for 
every $n\in\bN$, $((f^n)^{(m)})^*\delta_0/(d^n-m)=\delta_b\neq\mu_f$. 
An affine coordinate on $\bC$ is fixed in Theorem \ref{th:anyorder},
but note that
$A^*(((f^n)^{(m)})^*\delta_a-(d^n-m)\cdot\mu_f)=
(((A\circ f\circ A^{-1})^n)^{(m)})^*\delta_{(A')^{m-1}(a)}-(d^n-m)\cdot\mu_{A\circ f\circ A^{-1}}$
on $\bP^1$ for any affine transformation $A$ on $\bC$. 

The equidistribution \eqref{eq:equidistderivative}
for $m>1$ was expected in \cite[\S2.4]{Gauthier_Vigny_derivative}, 
at least when $f$ has no Siegel disks. 
As seen in the proof below, \eqref{eq:equidistderivative}
follows only by an analysis of $(f^n)^{(m)}$ 
on $\bP^1\setminus K(f)$ in this case. This analysis
is not difficult for $m=1$ by the chain rule, but for $m>1$ 
it requires to deal carefully with the higher order derivatives of the B\"ottcher coordinates of $f$ near $\infty$. An extra and more involved
effort is required to treat the situation on $K(f)$ under the presence of Siegel disks of $f$ in general.

\subsection{Over a non-archimedean complete valued field $K$}\label{sec:nonarchi}
Let $K$ be an algebraically closed field. We say that an absolute value $|\cdot|$ on $K$ is \emph{non-trivial} if
$|K|\not\subset\{0,1\}$ and that it is \emph{non-archimedean} if  
the {\em strong triangle inequality} 
$|z+w|\le\max\{|z|,|w|\}$ holds for any $z,w\in K$. 
For the details on the Berkovich projective line $\sP^1=\sP^1(K)$,
the canonical action of $f$ on $\sP^1$,
and the equilibrium (or canonical) measure $\mu_f$ of $f$ on $\sP^1$,
see Subsection \ref{sec:backgroundnonarchi} below.
By convention, we say $f$ {\em has no potentially good reductions}
if $\mu_f(\{\cS\})=0$ for any $\cS\in\sP^1\setminus\bP^1;$ this definition
coincides with the usual algebraic one $($cf.\ \cite[Corollary 10.33]{BR10}$)$. 

Our second principal result is a non-archimedean counterpart of Theorem \ref{th:anyorder}. 

\begin{mainth}\label{th:derivative}
Let $K$ be an algebraically closed field of characteristic $0$
that is complete with respect to a non-trivial and non-archimedean
absolute value. Let $m\in\bN$ and $f\in K[z]$ be a polynomial of degree $d>1$
having no potentially good reductions. Then for every $a\in K$, 
\begin{gather}
 \lim_{n\to\infty}\frac{\bigl((f^n)^{(m)}\bigr)^*\delta_a}{d^n-m}=\mu_f
\quad\text{weakly on }\sP^1.\label{eq:derivatives}
\end{gather}
\end{mainth}
The no potentially good reductions 
assumption allows us to deal with the Berkovich filled-in Julia set
$\sK(f)$ of $f$. The analysis on $\sP^1\setminus\sK(f)$ in the proof
is similar to that in the archimedean case, 
using  the (non-archimedean) B\"ottcher coordinate near $\infty$ and a
non-archimedean potential theory instead
(see \cite{Okuapripri}). 

\subsection{Over a product formula field $k$}\label{sec:adelic}
Let $k$ be a field. We denote by $\overline{k}$ an algebraic closure of $k$.
An effective $k$-divisor $\cZ$ on $\bP^1(\overline{k})$
is the scheme theoretic vanishing of some $P\in\bigcup_{d\in\bN}k[z_0,z_1]_d$.
Then, $\cZ$ is supported by $\overline{k}$ (regarding $\bP^1(\overline{k})$ 
as $\overline{k}\cup\{\infty\}$) if and only if $P(z_0,z_1)=z_0^{\deg p}p(z_1/z_0)$
for some $p(z)\in k[z]$ of degree $>0$ (identifying $[z_0:z_1]$ with $z_1/z_0$
when $z_0\neq 0$, that is, $\infty=[0:1]$ 
as the convention in \cite{FvdP04}), which is unique up to multiplication
in $k^*=k\setminus\{0\}$ and is called a {\em representative} of $\cZ$.

A field $k$ is a {\em product formula field} if $k$ is equipped with
a (possibly uncountable) family $M_k$ of (not necessarily all)
places of $k$, a family $(|\cdot|_v)_{v\in M_k}$ of 
non-trivial absolute values $|\cdot|_v$ representing $v$,
and a family $(N_v)_{v\in M_k}$ in $\bN$ satisfying the 
{\em product formula property} in that, for every $z\in k^*$,
\begin{center}
 $|z|_v=1$ for all but finitely many $v\in M_k$, and
 $\prod_{v\in M_k}|z|^{N_\nu}_v=1$.
\end{center}
A place $v\in M_k$ is said to be finite (resp.\ infinite)
if $|\cdot|_v$ is non-archimedean (resp.\ archimedean).
If $M_k$ contains an infinite place of $v$, then $k$ is (isomorphic to) a
number field (so there are at most finitely many infinite places of 
a product formula field). 
For each $v\in M_k$, let $k_v$ be the completion of $k$ with respect to
$|\cdot|_v$. Then $|\cdot|_v$ extends to $\overline{k_v}$.
Let $\bC_v$ be the completion of $\overline{k_v}$
with respect to $|\cdot|_v$ 
(so $|\cdot|_v$ extends to $\bC_v$) and fix an embedding of 
$\overline{k}$ to $\bC_v$ extending that of $k$ to $k_v$. 
By convention, the dependence of a 
{\em local} quantity induced by $|\cdot|_v$ on each $v\in M_k$ is 
emphasized by adding the suffix to it, like $k_v$ and $\bC_v$. 

Let $\hat{h}_f(\cZ)$ be the {\em Call-Silverman canonical height}
of an effective $k$-divisor $\cZ$ on $\bP^1(\overline{k})$
(see Subsection \ref{sec:height} below for the definition).
The following is our third principal result.

\begin{mainth}\label{th:adelicequidist}
Let $k$ be a product formula field of characteristic $0$, and
let $f\in k(z)$ be a polynomial of degree $d>1$ and $m\in\bN$. 
Then for every $a\in k$, denoting by $[(f^n)^{(m)}=a]$
the effective $k$-divisor on $\bP^1(\overline{k})$ 
whose representative is 
$(f^n)^{(m)}-a\in k[z]$,
%{\red [delete:taking into account their multiplicities]}, 
we have
the {\em $(g_{f,v})_{v\in M_k}$-small heights property}
\begin{gather}
 \lim_{n\to\infty}\hat{h}_f([(f^n)^{(m)}=a])=0\label{eq:heightvanishing}
\end{gather}
of the sequence $([(f^n)^{(m)}=a])_n$ of effective $k$-divisors 
on $\bP^1(\overline{k})$.

Assume, in addition, that $k$ is a number field and 
$a\in k^*$, then the uniform 
{\em asymptotically $(g_{f,v})_{v\in M_k}$-Fekete configuration property}
\begin{multline}
\lim_{n\to\infty}\sup_{v\in M_k}N_v
\int_{\sP^1 (\bC_v)\times\sP^1 (\bC_v)\setminus\diag_{\bP^1 (\bC_v)}}
(\log|\cS-\cS'|_v-g_{f,v}(\cS)-g_{f,v}(\cS'))\\
\biggl(\Bigl(\frac{\bigl((f^n)^{(m)}\bigr)^*\delta_a}{d^n-m}-\mu_{f,v}\Bigr)
\times\Bigl(\frac{\bigl((f^n)^{(m)}\bigr)^*\delta_a}{d^n-m}-\mu_{f,v}\Bigr)\biggr)(\cS,\cS')=0\label{eq:Fekete}
\end{multline}
of $([(f^n)^{(m)}=a])$ holds, so in particular, for every $v\in M_k$,
\begin{gather}
 \lim_{n\to\infty}\frac{\bigl((f^n)^{(m)}\bigr)^*\delta_a}{d^n-m}
=\mu_{f,v}\quad\text{weakly on }\sP^1(\bC_v).\label{eq:adelicequidist}
\end{gather}
\end{mainth}

The proof is based on an adelic equidistribution result 
for effective divisors on $\bP^1(\overline{k})$ 
having {\em small diagonals and} small heights (\cite{OkuDivisor}).

\subsection{The derivatives of the iterates of a H\'enon-type polynomial
automorphism of $\C^2$  }

Let $[t:z:w]$ be the homogeneous coordinate on $\P^2$,
endowed with the Fubini-Study form. 
Identifying $\C^2$ with $\{t=1\}$, we let
\begin{gather*}
 L_\infty:=\{t=0\}=\bP^2\setminus\bC^2 
\end{gather*}
be the {\em line at infinity} in $\bP^2$. We fix 
the orthonormal frame $(\partial_z,\partial_w)$ of the tangent space $T\bC^2$ of 
$\bC^2$,
so that for a polynomial endomorphism $f$ of $\bC^2$,
the derivative $\rd f$ of $f$ is identified 
with the $\mathrm{M}(2,\bC)$-valued function $(z,w)\mapsto(Df)_{(z,w)}$.
Here, a polynomial automorphism of $\bC^2$ is a polynomial endomorphism of $\bC^2$
whose inverse exists and is a polynomial endomorphism of $\bC^2$.

Recall some basic facts on a {\em H\'enon-type} 
polynomial automorphism $f$ of $\C^2$ 
of degree $d>1$ (\cite{Bed_Smi_critical,DS}).
The Jacobian determinant $J_f:=\det(Df)\in\bC[z,w]$ of $f$ 
is a non-zero constant on $\bC^2$, so for every $n\in\bN$, 
the Jacobian determinant $J_{f^n}=\det(D(f^n))\in\bC[z,w]$ of $f^n$
on $\bC^2$ is equal to the non-zero constant $J_{f}^n$. 
This $f$ extends to a birational self-map on $\P^2$, which is still 
denoted by $f$ for simplicity, so that both the indeterminacy loci $I^+,I^-$ 
of $f,f^{-1}$ are singletons in $L_\infty$, that
$I^-\neq I^+$ 
(so often normalized as $I^+=\{[0:0:1]\}, I^-=\{[0:1:0]\}$), 
and that $I^-=f(L_\infty\setminus I^+)$. 
Moreover, the unique point in $I^-$ is a superattracting fixed point 
of $f|(\bP^1\setminus I^+)$, 
and the attractive basin $B^+$ of $f|(\bP^1\setminus I^+)$ associate to $I^-$
satisfies $B^+\setminus\bC^2=L_{\infty}\setminus  I^+$.
Let $\|\cdot\|$ be the Euclidean norm on $\bC^2$. 
The {\em filled-in Julia set} of $f$ is defined by
\begin{gather*}
K^+:=\Bigl\{(z,w)\in \C^2: \limsup_{n \to \infty}\|f^n(z,w)\|<\infty\Bigr\}.
\end{gather*}
Then $\overline{K^+}= K^+\cup I^+$ in $\P^2$ and $\P^2=\overline{K^+}\cup B^+$
(see e.g.\ \cite[Proposition 5.5]{DS}).
The {\em Green function} $g^+$ of $f$ is the locally uniform limit
\begin{equation*}
g^+:= \lim_{n \to \infty}\frac{\log\max\{1,\|f^n\|\}}{d^n}\quad\text{on }\bC^2. 
\end{equation*} 
It is continuous and plurisubharmonic on $\C^2$, it is $>0$ and pluriharmonic 
on $B^+$, and it is $\equiv 0$ on $K^+$.
The {\em Green current} $T^+$ of $f$ is defined as the trivial extension of $\rd\rd^c g^+$ on $\bC^2$ to $\bP^2$. It is a positive closed $(1,1)$-current on $\P^2$ 
and moreover of mass $1$ (\cite[Lemma 6.3]{DS}).

For a non-constant polynomial $P\in\bC[z,w]$,
let $[P]$ be the current of integration along the hypersurface in $\bP^2$
defined by the zeros of (the homogenized) $P$ in $\bP^2$,
taking into account their multiplicities. The mass of $[P]$ 
equals $\deg P$ by B\'ezout's theorem. Let 
$I_2=\begin{pmatrix}
      1 & 0\\
      0 & 1
     \end{pmatrix}$ 
be the {\em identity matrix} in $\mathrm{M}(2,\bC)$. 

Our final principal result is the following.

\begin{mainth}\label{th:eigenvalue_automorphism}
 Let $f$ be a H\'enon-type polynomial automorphism of $\C^2$ of degree $d>1$ 
 and $\lambda \in \C^*=\bC\setminus\{0\}$. Then for every $n\in\bN$, 
 $\det(D(f^n)- \lambda I_2)\in \bC[z,w]$ is of degree $d^n-1$,
 and 
 \begin{gather}
  \lim_{n\to \infty}\frac{[\det(D(f^n)- \lambda I_2)]}{d^n-1}
  =T^+\quad\text{on }\bP^2\label{eq:equidistforward}
 \end{gather}
 as currents.
\end{mainth}
In the proof, we show the $L^1_{\mathrm{loc}}$-convergence of 
a sequence of potentials of $[\det(D(f^n)- \lambda I_2)]/(d^n-1)$ 
towards $g^+$ on $B^+$ as $n\to\infty$ using the first order partial derivatives 
of $g^+$. The pleasant {\em uniqueness} of $T^+$ among
all positive closed $(1,1)$-currents on $\bP^2$ of mass $1$ 
which are supported by $\overline{K^+}$ (\cite{Forn_Sib}; see also \cite[Theorem 6.5]{DS}) allows us to deal with $K^+$.

\subsection*{Organization of the article}
In Section \ref{sec:overC}, we treat the field $\bC$ of complex numbers.
In Subsection \ref{sec:background}, we recall some notion and facts 
from complex dynamics. In Subsection \ref{sec:complex}, we give a proof of
Theorem~\ref{th:anyorder} and, in Subsection \ref{sec:firstsecond}, we give 
a simpler treatment for the cases $m=1,2$. In Section \ref{sec:overnotC}, 
we treat a non-archimedean field $K$ and a product formula field $k$.
In Subsections \ref{sec:backgroundnonarchi} and \ref{sec:height},  
we recall a background from non-archimedean and 
arithmetic dynamics, respectively, and 
in Subsection \ref{sec:adelic2}, we show Theorems \ref{th:derivative} and
\ref{th:adelicequidist}. 
In Section \ref{sec:henon}, we show Theorem \ref{th:eigenvalue_automorphism}
in a slightly more general form.

\section{Proof of Theorem \ref{th:anyorder}}\label{sec:overC}
\subsection{Background from complex dynamics}
\label{sec:background}
Let $f\in\bC[z]$ be a polynomial of degree $d>1$. The superattractive basin 
\begin{gather*}
 I_\infty(f):=\Bigl\{z\in\bP^1:\lim_{n\to\infty}f^n(z)=\infty\Bigr\}
\end{gather*}
of $f$ associated to the superattracting fixed point $\infty$ 
of $f$ (regarding $\bP^1$ as $\bC\cup\{\infty\}$)
is a domain in $\bP^1$ containing $\infty$, and coincides with $\bP^1\setminus K(f)$.
Let $C(f)$ be the critical set of $f$ (as a branched self-covering of $\bP^1$) 
which consists of $\infty$ and all the  zeros of $f'$ on $\bC$. The set
$\bigcup_{n\in\bN\cup\{0\}}f^{-n}(C(f)\setminus\{\infty\})$ is bounded in $\bC$.

The topology of $\bP^1$ coincides with the induced one
from the chordal metric on $\bP^1$.
The Julia set $J(f)$ of $f$ is defined as the set of all $z\in\bP^1$
at which the family $(f^n:\bP^1\to\bP^1)_{n\in\bN}$ is not normal. 
The Fatou set $F(f)$ of $f$ is defined by $\bP^1\setminus J(f)$
and a component of $F(f)$ is called a Fatou 
component of $f$. Both $J(f)$ and $F(f)$ are totally invariant
under $f$ and 
%{\red [delete:$J(f)$ coincides with the boundary of the filled-in Julia set]} 
\begin{gather*}
 J(f)=\partial K(f)=\partial I_\infty(f). 
\end{gather*}
Any Fatou component of $f$ is either $I_\infty(f)$ or a component 
of the interior of $K(f)$ 
and is mapped properly to a Fatou component of $f$. Any Fatou component
of $f$ other than $I_\infty(f)$ is simply connected.
A Fatou component $W$ of $f$
is said to be \emph{cyclic} under $f$ if there is $p\in\bN$ such that $f^p(W)=W$.
If in addition the restriction $f^p:W\to W$ is injective, $W$ 
is called a {\em Siegel disk} of $f$ and then there exists a holomorphic
injection $h:W\to\bC$ such that for some $\alpha\in\bR\setminus\bQ$,
$h\circ f^p=e^{2i\pi\alpha}\cdot h$ on $W$. In particular, $h(V)=\{|w|<r\}$ for some $r>0$, $v_0:=h^{-1}(0)$ is fixed by $f^p$, and 
$(f^p)'(v_0)=e^{2i\pi\alpha}$. 
For more details on complex dynamics, see e.g.\ \cite{Milnor4}.

\subsection{Proof of Theorem \ref{th:anyorder}}\label{sec:complex} 

Let $f\in\bC[z]$ be a polynomial of degree $d>1$. Fix $m\in\bN$. 
\begin{lemma}\label{th:basin}
We have
\begin{gather}
 (f^n)^{(m)}=\bigl((e^{O(1)}\cdot d^n)^m+O(d^{(m-1)n})\bigr)\cdot f^n\quad\text{as }n\to\infty\label{eq:higherorderderivative}
\end{gather}
on $I_\infty(f)\setminus\bigcup_{n\in\bN\cup\{0\}}f^{-n}(C(f))$ locally uniformly.
Moreover, for every $a\in\bC$, the family
$\bigl((\log|(f^n)^{(m)}-a|)/(d^n-m)\bigr)_n$ 
of subharmonic functions on $\bC$
is locally uniformly bounded from above on $\bC$ and 
\begin{gather}
\lim_{n\to\infty}\frac{\log|(f^n)^{(m)}-a|}{d^n-m}=g_f\label{eq:aroundinfty}
\end{gather}
locally uniformly on 
$I_\infty(f)\setminus\bigcup_{n\in\bN\cup\{0\}}f^{-n}(C(f))$.
\end{lemma}

\begin{proof}
Fixing $r\gg 1$, there exists a biholomorphism
$w=\psi(z)$ from $\bP^1\setminus\{g_f\le r\}$
to $\bP^1\setminus\{|w|\le e^r\}$, 
which is called a {\em B\"ottcher coordinate}
near $\infty$ associated to $f$, such that 
$\psi(f(z))=\psi(z)^d$ on $\bP^1\setminus\{g_f\le r\}$. 
Then 
$\psi(\infty)=\infty$, $\psi'\neq 0$ on $\bC\setminus\{g_f\le r\}$,
and letting $\iota:\bP^1\to\bP^1$ be the involution $z\mapsto 1/z$ (regarding $1/0$ as $\infty$), $(\iota\circ\psi\circ\iota)'(0)=1/(\iota\circ\psi^{-1}\circ\iota)'(0)\neq 0$. 

We first claim that 
\begin{gather}
\frac{(f^n)'}{f^n}(z)
=d^n\cdot\bigl(1+O(\psi(z)^{-d^n})\bigr)\cdot\frac{\psi'}{\psi}(z)\quad\text{as }n\to\infty\label{eq:nearinfty} 
\end{gather}
on $\bC\setminus\{g_f\le r\}$ uniformly; indeed, 
for every $n\in\bN$, since 
$\psi(f^n(z))=\psi(z)^{d^n}$ on $\bC\setminus\{g_f\le r\}$,
we have
$f^n(z)=\psi^{-1}(\psi(z)^{d^n})$ 
and $\psi'(f^n(z))\cdot(f^n)'(z)=d^n\cdot\psi(z)^{d^n-1}\cdot\psi'(z)$
on $\bC\setminus\{g_f\le r\}$, 
so that
\begin{gather*}
\frac{(f^n)'(z)}{f^n(z)}
=\frac{d^n\cdot\psi(z)^{d^n-1}\cdot\psi'(z)}{\psi^{-1}\bigl(\psi(z)^{d^n}\bigr)\cdot\psi'(f^n(z))}
= d^n\cdot\frac{\frac{\psi(z)^{d^n}}{\psi^{-1}(\psi(z)^{d^n})}}{\psi'(f^n(z))} \cdot \frac{\psi'(z)}{\psi(z)}
\end{gather*}
on $\bC\setminus\{g_f\le r\}$. Moreover, we have
\begin{align*}
\frac{\psi(z)^{d^n}}{\psi^{-1}(\psi(z)^{d^n})}
=&\frac{(\iota\circ\psi^{-1}\circ\iota)(1/\psi(z)^{d^n})-(\iota\circ\psi^{-1}\circ\iota)(0)}{1/\psi(z)^{d^n}-0}
=(\iota\circ\psi^{-1}\circ\iota)'(0)+O(1/\psi(z)^{d^n})\\
=&\frac{1}{(\iota\circ\psi\circ\iota)'(0)}+O(\psi(z)^{-d^n})
%={\red \frac{1}{(\iota\circ\psi\circ\iota)'(0)}+O(z^{-d^n})}
\quad\text{as }n\to\infty
\end{align*}
on $\bC\setminus\{g_f\le r\}$ uniformly and, 
since
$(\iota\circ\psi\circ \iota)'(1/f^n(z))=-\frac{\psi'(f^n(z))\cdot\{-(f^n(z)^2)\}}{\psi(f^n(z))^2}$ on $\bC\setminus\{g_f\le r\}$ by the chain rule, we also have
\begin{align*}
\psi'(f^n(z))
=&
%(\iota\circ\psi\circ\iota)'(1/f^n(z))\cdot\psi(f^n(z))^2\cdot(1/f^n(z))^2=
\frac{(\iota\circ\psi\circ\iota)'(1/f^n(z))}{\bigl(\frac{(\iota\circ\psi\circ\iota)(1/f^n(z))}{1/f^n(z)}\bigr)^2}
=\frac{(\iota\circ\psi\circ\iota)'(0)+
((\iota\circ\psi\circ\iota)'(1/f^n(z))-(\iota\circ\psi\circ\iota)'(0))}{\bigl(\frac{(\iota\circ\psi\circ\iota)(1/f^n(z))-(\iota\circ\psi\circ\iota)(0)}{1/f^n(z)-0}\bigr)^2}\\
=&\frac{(\iota\circ\psi\circ\iota)'(0)+O(1/f^n(z))}{\bigl((\iota\circ\psi\circ\iota)'(0)+O(1/f^n(z))\bigr)^2}\\
=&\frac{1}{(\iota\circ\psi\circ\iota)'(0)}+O(1/f^n(z))
=\frac{1}{(\iota\circ\psi\circ\iota)'(0)}+O(\psi(z)^{-d^n})
\quad\text{as }n\to\infty
\end{align*}
on $\bC\setminus\{g_f\le r\}$ uniformly. Hence the claim holds.

For any domain $D\Subset I_\infty(f)\cap K$ 
and any $M\in\bN\cup\{0\}$ so large that
$f^M(D)\subset\bP^1\setminus\{g_f\le r\}$, by \eqref{eq:nearinfty},
we have
\begin{gather*}
\frac{(f^n)'}{f^n}=\frac{((f^{n-M})'\circ f^M)\cdot(f^M)'}{f^{n-M}\circ f^M}
=d^{n-M}\cdot\Bigl(\frac{\psi'}{\psi}\circ f^M\cdot(f^M)'\Bigr)+o(1) 
\quad\text{as }n\to\infty
\end{gather*}
on some open neighborhood of $\overline{D}$ uniformly. 
Let us show by induction 
{\em that for any $m\in\bN$,
\begin{gather}
\frac{(f^n)^{(m)}}{f^n}
=\biggl(d^{n-M}\cdot\frac{\psi'}{\psi}\circ f^M\cdot(f^M)'\biggr)^m
+O(d^{(m-1)n})\quad\text{as }n\to\infty\label{eq:logarithmic}
\end{gather}
on some open neighborhood of $\overline{D}$ uniformly};
we have just seen \eqref{eq:logarithmic} for $m=1$
on some open neighborhood of $\overline{D}$ uniformly, so 
assume that $m>1$ and that \eqref{eq:logarithmic} for $m-1$
holds on some open neighborhood of $\overline{D}$ uniformly.
Then 
using Cauchy's estimate,
we have
\begin{gather*}
 \frac{(f^n)^{(m)}}{f^n} - \frac{(f^n)^{(m-1)}\cdot (f^n)' }{f^n\cdot f^n} 
=\biggl(\frac{(f^n)^{(m-1)}}{f^n}\biggr)'
=O(d^{n(m-1)})\quad\text{as }n\to\infty    
\end{gather*}
on some open neighborhood of
$\overline{D}$ uniformly, which with \eqref{eq:logarithmic} for 
both $1$ and $m-1$ 
on some open neighborhood of $\overline{D}$ uniformly
yields
\begin{multline*}
\frac{(f^n)^{(m)}}{f^n}
= \frac{(f^n)^{(m-1)}\cdot (f^n)' }{f^n\cdot f^n} +O(d^{(m-1)n})  \\
= \Bigl(\Bigl(d^{n-M}\cdot\frac{\psi'}{\psi}\circ f^M\cdot(f^M)'\Bigr)^{m-1}
+O(d^{(m-2)n})\Bigr)\cdot\Bigl(d^{n-M}\cdot\frac{\psi'}{\psi}\circ f^M\cdot(f^M)'
+O(1)\Bigr)\\
+ O(d^{(m-1)n})\quad\text{as }n\to\infty 
\end{multline*}
on some open neighborhood of $\overline{D}$ uniformly.
This yields \eqref{eq:logarithmic} for $m$ 
on some open neighborhood of $\overline{D}$ uniformly
and concludes the induction. 
Now, if in addition 
$D\Subset I_\infty(f)\setminus\bigcup_{n\in\bN\cup\{0\}}f^{-n}(C(f))$, 
so $\inf_D|\frac{\psi'}{\psi}\circ f^M\cdot(f^M)'|>0$, then
the estimate \eqref{eq:logarithmic} 
yields the asymptotic estimate \eqref{eq:higherorderderivative}. 

Fix $a\in\bC$. The final locally uniform convergence \eqref{eq:aroundinfty}
follows from \eqref{eq:higherorderderivative} and \eqref{eq:Greeniterate}. Then, for every $R>0$ so large that 
$\bigcup_{n\in\bN\cup\{0\}}f^{-n}(C(f)\setminus\{\infty\})\subset\{|z|<R\}$,
we also have
\begin{gather*}
\frac{\log|(f^n)^{(m)}-a|}{d^n-m}
\le\frac{\log(2\max\{|(f^n)^{(m)}|,|a|\})}{d^n-m}
\le g_f+O(1)\quad\text{as }n\to\infty
\end{gather*}
on $\{|z|=R\}$ uniformly. Hence by
the maximum principle for subharmonic functions, we deduce that 
the family $((\log|(f^n)^{(m)}-a|)/(d^n-m))_n$ 
is locally uniformly bounded from above on $\bC$.
\end{proof}

\begin{remark}[the Schwarzian and pre-Schwarzian derivatives $S_{f^n},T_{f^n}$
of $f^n$]
The expression of $(f^n)^{(m)}$ given by \eqref{eq:logarithmic}
in the proof of Lemma \ref{th:basin}
also quantifies Ye \cite[Theorems 1.1 and 3.3]{Ye11} as
\begin{gather*}
 S_{f^n}:=\frac{(f^n)'''}{(f^n)'}-\frac{3}{2}\Bigl(\frac{(f^n)''}{(f^n)'}\Bigr)^2
=-2d^{2n}\cdot(\partial_zg_f)^2+O(d^n)\quad\text{and}\\
T_{f^n}:=\frac{(f^n)''}{(f^n)'}=2d^n\cdot\partial_zg_f+O(1)
\quad\text{as }n\to\infty 
\end{gather*}
on $I_\infty(f)\setminus\bigcup_{n\in\bN\cup\{0\}}f^{-n}(C(f))$ locally uniformly.
Indeed, recall that $g_f=\log|\psi|$ so 
$\partial_zg_f=\psi'/(2\psi)$
on $\bC\setminus\{g_f\le r\}$, and  $g_f\circ f=d\cdot g_f$
so $(\partial_zg_f)\circ f^M\cdot(f^M)'=d^M\cdot\partial_z g_f$
on $I_\infty(f)$. Hence \eqref{eq:logarithmic} is rewritten as
\begin{align*}
(f^n)^{(m)}
&=\bigl(\bigl(d^{n-M}\cdot(2\partial_z g_f)\circ f^M\cdot (f^M)'\bigr)^m
+O(d^{(m-1)n})\bigr)\cdot f^n\\
&=\bigl(\bigl(2d^n\cdot\partial_z g_f)^m+O(d^{(m-1)n})\bigr)\cdot f^n
\quad\text{as }n\to\infty
\end{align*}
on $\overline{D}$ uniformly. 
This for $m\in\{1,2,3\}$ yields the above asymptotics of $S_{f^n}$ and $T_{f^n}$.
\end{remark}

Fix $a\in\bC$, and let us continue the proof of Theorem \ref{th:anyorder}.
By the final two assertions in Lemma \ref{th:basin},
applying to $((\log|(f^n)^{(m)}-a|)/(d^n-m))_n$
a {\em compactness principle} (see \cite[Theorem 4.1.9(a)]{Hormander83})
for a family of subharmonic functions on a domain in $\bR^N$,
there are a sequence
$(n_j)$ in $\bN$ tending to $+\infty$ as $j\to\infty$ and 
a subharmonic function $\phi$ on $\bC$ such that
\begin{gather}
 \phi:=\lim_{j\to\infty}\frac{\log|(f^{n_{j}})^{(m)}-a|}{d^{n_{j}}-m}
\quad\text{in }L^1_{\operatorname{loc}}(\bC,m_2) 
\label{eq:subsequence}
\end{gather}
($m_2$ denotes the (real $2$-dimensional) Lebesgue measure on $\bC$). 
By \eqref{eq:aroundinfty},
we have $\phi\equiv g_f$ $m_2$-a.e.\ on 
$I_\infty(f)\setminus\bigcup_{n\in\bN\cup\{0\}}f^{-n}(C(f))$,
and in turn on $I_\infty(f)$ 
by the subharmonicity of $\phi-g_f$ on $I_\infty(f)\cap\bC$. 
Then also by $I_\infty(f)=\{g_f>0\}$, the subharmonicity of $\phi$ on $\bC$, 
and the maximum principle for subharmonic functions,
we have $\phi\le\max_{\{g_f=\epsilon\}}\phi=\max_{\{g_f=\epsilon\}}g_f=\epsilon$ 
on $K(f)=\{g_f=0\}\subset\{g_f<\epsilon\}$ for every $\epsilon>0$,
and in turn $\phi\le 0$ on $K(f)$.
By the upper semicontinuity
of $\phi-g_f$ on $\bC$, the subset $\{\phi<g_f\}$ is open in $\bC$.

\begin{lemma}\label{th:limit}
If $a\neq 0$, then $\phi= g_f$ on $\bC$.
\end{lemma}

\begin{proof}
Suppose that $\{\phi<g_f\}\neq\emptyset$ and let us show $a=0$. 
By $\phi\equiv g_f$ on $I_\infty(f)$, there is a Fatou component $U\subset K(f)$
of $f$ containing a component $W$ of $\{\phi<g_f\}$. 
Since $\phi\le g_f=0$ on $U$, we in fact have $U=W$
by the maximum principle for subharmonic functions.

\subsection*{(I)}
Taking a subsequence of $(n_j)$ if necessary, there is a locally uniform limit 
\begin{gather*}
 g:=\lim_{j\to\infty}f^{n_j}\quad\text{on }U. 
\end{gather*}
We claim that 
\begin{gather*}
 g^{(m)}\equiv a
\end{gather*}
on $U$, so in particular we can say $g\in\bC[z]$ (of degree $\le m$); indeed, for any domain $D\Subset U=W$, by 
Hartogs's lemma for a sequence of
subharmonic functions on a domain in $\bR^N$ 
(see \cite[Theorem 4.1.9(b)]{Hormander83}),
we have 
\begin{gather}
 \limsup_{j\to\infty}\sup_{\overline{D}}\frac{\log|(f^{n_j})^{(m)}-a|}{d^{n_j}-m}
\le\sup_{\overline{D}}\phi<0.\label{eq:Hartogs} 
\end{gather}
Then
$g^{(m)}=\bigl(\lim_{j\to\infty}(f^{n_j})\bigr)^{(m)}
=\lim_{j\to\infty}\bigl((f^{n_j})^{(m)}\bigr)\equiv a$ on $D$,
so the claim holds.
In the case that $g$ is constant, we have $g^{(m)}\equiv 0=a$, so we are done.

\subsection*{(II)} 
Let us assume that $g$ is non-constant.
Then by Hurwitz's theorem and Fatou's 
classification of cyclic Fatou components of $f$
(see, e.g., \cite[\S16]{Milnor4}),
there is $N\in\bN$ such that $V:=f^{n_N}(U)=g(U)(\supset g(\overline{D}))$
is a Siegel disk of $f$. Setting $p:=\min\{n\in\bN:f^n(V)=V\}$,
for any $j\ge N$, we have $p|(n_j-n_N)$ 
and there is a holomorphic injection
$h:V\to\bC$ such that for some $\alpha\in\bR\setminus\bQ$, setting
$\lambda:=e^{2i\pi\alpha}\in\partial\bD$, we have
$h\circ f^p=\lambda\cdot h$ on $V$. Hence for every $j\ge N$,
\begin{gather}
h\circ f^{n_j}=\lambda^{(n_j-n_N)/p}\cdot(h\circ f^{n_N})\quad\text{on }U.\label{eq:functional} 
\end{gather}
Taking a subsequence of $(n_j)$ if necessary, the limit
\begin{gather*}
 \lambda_0:=\lim_{j\to\infty}\lambda^{(n_j-n_N)/p}\in\partial\bD
\end{gather*}
also exists and then
\begin{gather}
 h\circ g=\lambda_0\cdot(h\circ f^{n_N})\quad\text{on }U\tag{\ref{eq:functional}$'$}.\label{eq:functionallimit} 
\end{gather}
Set $v_0:=h^{-1}(0)$ and fix $z_0\in U\cap f^{-n_N}(v_0)$, so that 
$f^p(v_0)=v_0=g(z_0)$ and $(f^p)'(v_0)=\lambda$.
For every $0<r\ll 1$, $\{|w|<2r\}\Subset h(V)$, and
letting $D_r$ be a component of
$(h\circ f^{n_N})^{-1}(\{|w|<r\})$ containing $z_0$, 
the restriction $h\circ f^{n_N}:D_r\setminus\{z_0\}\to\{0<|w|<r\}$
is an unramified covering of degree $\deg_{z_0}(f^{n_N})=\deg_{z_0}g$. Hence, 
the restriction 
$h\circ g:D_r\setminus\{z_0\}\to\{0<|w|<r\}$ is also an unramified covering 
of the same degree as that of $h\circ f^{n_N}|D_r$
by Hurwitz's theorem.
Let us denote by $h^{-1}$ the holomorphic inverse of 
the biholomorphism $h:V\to h(V)\subset\bC$.

Let us see by induction that for any $\ell\in\bN$, 
\begin{gather}
\bigl((h^{-1})^{(\ell)}(\lambda_0\cdot h\circ f^{n_N}(z))\cdot
(h\circ f^{n_N}(z))^\ell\bigr)^{(m)}
\equiv 0\quad\text{on }D_r;\label{eq:key}
\end{gather}
indeed, for every $j\ge N$, applying Cauchy's integration formula to
$f^{n_j}-g$ on $D_r$, by 
$g^{(m)}\equiv a$, \eqref{eq:functional}, and \eqref{eq:functionallimit},
we have
\begin{align}
\notag&\frac{(f^{n_j})^{(m)}(z)-a}{m!}
%=\frac{(f^{n_j}-g)^{(m)}(z)}{m!(\lambda^{(n_j-n_N)/p}-\lambda_0)}
=\int_{\partial D_r}\frac{f^{n_j}(\zeta)-g(\zeta)}{(\zeta-z)^{m+1}}\frac{\rd\zeta}{2i\pi}\\
=&\int_{\partial D_r}\frac{h^{-1}(\lambda^{(n_j-n_N)/p}\cdot h\circ f^{n_N}(\zeta))
-h^{-1}(\lambda_0\cdot h\circ f^{n_N}(\zeta))}{(\zeta-z)^{m+1}}\frac{\rd\zeta}{2i\pi}\label{eq:difference}
\\
\notag=&(\lambda^{(n_j-n_N)/p}-\lambda_0)\cdot
\int_{\partial D_r}\frac{
\frac{h^{-1}(\lambda^{(n_j-n_N)/p}\cdot h\circ f^{n_N}(\zeta))
-h^{-1}(\lambda_0\cdot h\circ f^{n_N}(\zeta))}{\lambda^{(n_j-n_N)/p}\cdot h\circ f^{n_N}(\zeta)-\lambda_0\cdot h\circ f^{n_N}(\zeta)}\cdot(h\circ f^{n_N}(\zeta))}{(\zeta-z)^{m+1}}\frac{\rd\zeta}{2i\pi}\\ 
\notag=&(\lambda^{(n_j-n_N)/p}-\lambda_0)\\
\notag&\times\int_{\partial D_r}\frac{\bigl((h^{-1})'(\lambda_0\cdot h\circ f^{n_N}(\zeta))+O\bigl(\lambda^{(n_j-n_N)/p}-\lambda_0\bigr)\bigr)\cdot
(h\circ f^{n_N}(\zeta))}{(\zeta-z)^{m+1}}\frac{\rd\zeta}{2i\pi}\quad\text{as }j\to\infty
\end{align}
on $D_r$, where recalling $h\circ f^{n_N}(\partial D_r)=\{|w|=r\}$
and $\{|w|<2r\}\Subset h(V)$
and applying Cauchy's estimate 
to the holomorphic function $h^{-1}|\{w'\in\bC:|w'-w|\le r\}$ for each $|w|=r$,
%and the maximum modulus principle to
%the holomorphic function $h^{-1}|\{|w|\le 2r\}$,
the $O(\lambda^{(n_j-n_N)/p}-\lambda_0)$ term is estimated as
\begin{align*}
&|O\bigl(\lambda^{(n_j-n_N)/p}-\lambda_0\bigr)|\\
\le&\sum_{k=2}^\infty\frac{|(h^{-1})^{(k)}(\lambda_0\cdot h\circ f^{n_N}(\zeta))|}{k!}|\lambda^{(n_j-n_N)/p}\cdot h\circ f^{n_N}(\zeta)-\lambda_0\cdot h\circ f^{n_N}(\zeta)|^{k-1}\\
\le&\sum_{k=2}^\infty\frac{\max_{|w|=r}|(h^{-1})^{(k)}(w)|}{k!}
(|\lambda^{(n_j-n_N)/p}-\lambda_0|\cdot r)^{k-1}\\
\le&\sum_{k=2}^\infty
\frac{\max_{|w|=2r}|h^{-1}(w)|}{r^k}(|\lambda^{(n_j-n_N)/p}-\lambda_0|\cdot r)^{k-1}\\
=&\frac{\max_{|w|=2r}|h^{-1}(w)|}{r}\cdot\frac{|\lambda^{(n_j-n_N)/p}-\lambda_0|}{1-|\lambda^{(n_j-n_N)/p}-\lambda_0|}
\quad\text{on }\partial D_r
\end{align*}
so the implicit constant of it
is independent of $z\in D_r$ and $\zeta\in\partial D_r$.
On the other hand, for every $z\in D_r$, by \eqref{eq:Hartogs} 
and \cite[(3.8)]{OkuyamaDerivatives}, we also have
\begin{gather}
 \limsup_{j\to\infty}\frac{\log|(f^{n_j})^{(m)}(z)-a|}{d^{n_j}-m}<0
\quad\text{and}\quad
\lim_{j\to\infty}\frac{\log|\lambda^{(n_j-n_N)/p}-\lambda_0|}{d^{n_j}-m}=0.\label{eq:apriorihigher}
\end{gather}
Hence also by Cauchy's integration formula, we have 
\begin{multline*}
\bigl((h^{-1})'(\lambda_0\cdot h\circ f^{n_N}(z))\cdot
h\circ f^{n_N}(z)\bigr)^{(m)}
=m!\int_{\partial D_r}\frac{(h^{-1})'(\lambda_0\cdot h\circ f^{n_N}(\zeta))\cdot
(h\circ f^{n_N}(\zeta))}{(\zeta-z)^{m+1}}\frac{\rd\zeta}{2i\pi}
\equiv 0
\end{multline*}
on $D_r$, that is, \eqref{eq:key} holds for $\ell=1$.

Next, suppose that \eqref{eq:key} holds for $1,\ldots,\ell-1$. Then
applying Cauchy's integration formula to $((h^{-1})^{(k)}(\lambda_0\cdot h\circ f^{n_N}(z))\cdot(h\circ f^{n_N}(z))^k)^{(m)}\equiv 0$ on $D_r$
for $k\in\{1,\ldots,\ell-1\}$, also by \eqref{eq:difference}, we have 
\begin{align*}
&\frac{(f^{n_j})^{(m)}(z)-a}{m!}\\
=&\frac{(f^{n_j})^{(m)}(z)-a}{m!}
-\sum_{k=1}^{\ell-1}(\lambda^{(n_j-n_N)/p}-\lambda_0)^k\cdot\frac{\bigl((h^{-1})^{(k)}(\lambda_0\cdot h\circ f^{n_N}(z))\cdot(h\circ f^{n_N}(z))^k\bigr)^{(m)}}{m!k!}\\
=&\int_{\partial D_r}\frac{h^{-1}(\lambda^{(n_j-n_N)/p}\cdot h\circ f^{n_N}(\zeta))
-h^{-1}(\lambda_0\cdot h\circ f^{n_N}(\zeta))}{(\zeta-z)^{m+1}}\frac{\rd\zeta}{2i\pi}\\
&-\sum_{k=1}^{\ell-1}(\lambda^{(n_j-n_N)/p}-\lambda_0)^k\cdot\int_{\partial D_r}\frac{\frac{1}{k!}(h^{-1})^{(k)}(\lambda_0\cdot h\circ f^{n_N}(\zeta))\cdot
(h\circ f^{n_N}(\zeta))^k}{(\zeta-z)^{m+1}}\frac{\rd\zeta}{2i\pi}\\
=&\int_{\partial D_r}\frac{\sum_{k=\ell}^{\infty}\frac{1}{k!}
(h^{-1})^{(k)}(\lambda_0\cdot h\circ f^{n_N}(\zeta))\cdot
\bigl(\lambda^{(n_j-n_N)/p}\cdot h\circ f^{n_N}(\zeta)-\lambda_0\cdot h\circ f^{n_N}(\zeta)\bigr)^k}{(\zeta-z)^{m+1}}\frac{\rd\zeta}{2i\pi}\\
=&(\lambda^{(n_j-n_N)/p}-\lambda_0)^\ell\times\\
&\times\int_{\partial D_r}\frac{
\frac{\sum_{k=\ell}^{\infty}\frac{1}{k!}
(h^{-1})^{(k)}(\lambda_0\cdot h\circ f^{n_N}(\zeta))\cdot
(\lambda^{(n_j-n_N)/p}\cdot h\circ f^{n_N}(\zeta)-\lambda_0\cdot h\circ f^{n_N}(\zeta))^k}{(\lambda^{(n_j-n_N)/p}\cdot h\circ f^{n_N}(\zeta)-\lambda_0\cdot h\circ f^{n_N}(\zeta))^\ell}\cdot\bigl(h\circ f^{n_N}(\zeta)\bigr)^\ell}{(\zeta-z)^{m+1}}\frac{\rd\zeta}{2i\pi}\\
=&(\lambda^{(n_j-n_N)/p}-\lambda_0)^\ell\times\\
&\times\int_{\partial D_r}
\frac{\frac{1}{\ell!}\bigl((h^{-1})^{(\ell)}(\lambda_0\cdot h\circ f^{n_N}(\zeta))+O\bigl(\lambda^{(n_j-n_N)/p}-\lambda_0\bigr)\bigr)\cdot
\bigl(h\circ f^{n_N}(\zeta)\bigr)^\ell}{(\zeta-z)^{m+1}}\frac{\rd\zeta}{2i\pi}\quad\text{as }j\to\infty
\end{align*}
on $D_r$, where
recalling $h\circ f^{n_N}(\partial D_r)=\{|w|=r\}$
and $\{|w|<2r\}\Subset h(V)$ 
and applying Cauchy's estimate to
the holomorphic function $h^{-1}|\{w'\in\bC:|w'-w|\le r\}$ for each $|w|=r$,
%and the maximum modulus principle to
%the holomorphic function $h^{-1}|\{|w|\le 2r\}$,
the $O(\lambda^{(n_j-n_N)/p}-\lambda_0)$ term is estimated as
\begin{align*}
&|O\bigl(\lambda^{(n_j-n_N)/p}-\lambda_0\bigr)|\\
\le&\sum_{k=\ell+1}^\infty\frac{|(h^{-1})^{(k)}(\lambda_0\cdot h\circ f^{n_N}(\zeta))|}{k!}|\lambda^{(n_j-n_N)/p}\cdot h\circ f^{n_N}(\zeta)-\lambda_0\cdot h\circ f^{n_N}(\zeta)|^{k-\ell}\\
\le&\sum_{k=\ell+1}^\infty\frac{\max_{|w|=r}|(h^{-1})^{(k)}(w)|}{k!}(|\lambda^{(n_j-n_N)/p}-\lambda_0|\cdot r)^{k-\ell}\\
\le&\sum_{k=\ell+1}^\infty
\frac{\max_{|w|=2r}|h^{-1}(w)|}{r^k}(|\lambda^{(n_j-n_N)/p}-\lambda_0|\cdot r)^{k-\ell}\\
=&\frac{\max_{|w|=2r}|h^{-1}(w)|}{r^\ell}\cdot\frac{|\lambda^{(n_j-n_N)/p}-\lambda_0|}{1-|\lambda^{(n_j-n_N)/p}-\lambda_0|}
\quad\text{on }\partial D_r 
\end{align*}
so the implicit constant of it 
is independent of $z\in D_r$ and $\zeta\in\partial D_r$. Hence by \eqref{eq:apriorihigher} again, 
also using  Cauchy's integration formula, we have
\begin{multline*}
\bigl((h^{-1})^{(\ell)}(\lambda_0\cdot h\circ f^{n_N}(z))\cdot
(h\circ f^{n_N}(z))^\ell\bigr)^{(m)}\\
=m!\int_{\partial D_r}\frac{(h^{-1})^{(\ell)}(\lambda_0\cdot h\circ f^{n_N}(\zeta))\cdot(h\circ f^{n_N}(\zeta))^\ell}{(\zeta-z)^{m+1}}\frac{\rd\zeta}{2i\pi}
\equiv 0\quad\text{on }D_r,
\end{multline*}
that is, \eqref{eq:key} holds for $\ell$ and concludes the induction.

Once this claim \eqref{eq:key} is at our disposal, for every $\ell\in\bN$,
there is $P_\ell\in\bC[z]$ of degree $<m$ such that
\begin{gather*}
 (h^{-1})^{(\ell)}(\lambda_0\cdot h\circ f^{n_N}(z))\cdot
(h\circ f^{n_N}(z))^\ell\equiv P_\ell(z)\quad\text{on }D_r. 
\end{gather*}
Then recalling $(h\circ f^{n_N})(z_0)=0$, 
for every $\ell\ge m$, we have
$P_\ell\equiv P_\ell(z_0)=0$; for, otherwise, 
we must have $m>\deg P_\ell\ge\deg_{z_0}P_\ell\ge\ell\ge m$, 
which is a contradiction.
Consequently, 
also by \eqref{eq:functionallimit}
and $(h\circ f^{n_N})(D_r\setminus\{z_0\})=\{0<|w|<r\}$, 
for every $\ell\ge m$, 
\begin{gather*}
 (h^{-1})^{(\ell)}\bigl((h\circ g)(z)\bigr)
=(h^{-1})^{(\ell)}\bigl(\lambda_0\cdot h\circ f^{n_N}(z)\bigr)\equiv 0
\quad\text{on }D_r,
\end{gather*}
which implies that there is $Q\in\bC[z]$ (of degree $<m$) such that
$h^{-1}\equiv Q$ on $\{0<|w|<r\}$ 
since $h\circ g:D_r\setminus\{z_0\}\to\{0<|w|<r\}$
is an unramified covering.
Then $\deg Q>0$.

On the other hand, we also have
\begin{gather*}
 f^p(Q(w))=f^p(h^{-1}(w))=h^{-1}(\lambda w)=Q(\lambda w)\quad\text{on }\{0<|w|<r\},
\end{gather*}
and in turn $f^p(Q(w))=Q(\lambda w)$ in $\bC[w]$
by the identity theorem for holomorphic functions.
Then $Q\in\bC[w]$ must be constant since $\deg(f^p)=d^p>1$. 
This contradicts $\deg Q>0$.

Hence $g$ is constant, and the proof of Lemma \ref{th:limit} is complete.
\end{proof}

Using Lemma~\ref{th:limit},
the $L^1_{\operatorname{loc}}(\bC,m_2)$-convergence 
\eqref{eq:subsequence}, 
a continuity of the Laplacian $\Delta$, and the equalities
\begin{gather*}
\Delta\frac{\log|(f^{n_j})^{(m)}-a|}{d^{n_j}-m}=\frac{\bigl((f^{n_j})^{(m)}\bigr)^*\delta_a}{d^{n_j}-m}
\quad\text{on }\bC
\end{gather*}
for each $j\in\bN$ and $\Delta g_f=\mu_f$ on $\bC$, 
whenever $a\in\bC\setminus\{0\}$, 
we conclude the
desired weak convergence \eqref{eq:equidistderivative} on $\bC$,
and in turn on $\bP^1$ since $\supp\mu_f\subset\bC$. 
Now the proof of Theorem \ref{th:anyorder} is complete. 
\qed

\subsection{On the proof of Theorem \ref{th:anyorder} for the first and second orders derivatives}
\label{sec:firstsecond}

In step (II) of the proof of Lemma \ref{th:limit} in Section \ref{sec:background},
it might be interesting 
to show that $a=0$ 
by direct computations in the case where $g$ is non-constant,
instead of showing that $g$ is constant by contradiction. 
We include herewith such proofs in (II)' and (II)'' below for the first and second orders derivatives cases $m=1,2$, respectively.

\subsection*{(II)'} Here, assume that $m=1$ and that $g$ is non-constant.
For any $j\ge N$, differentiating both sides in \eqref{eq:functional},
by the chain rule, we have
\begin{gather*}
 (h'\circ f^{n_j})\cdot(f^{n_j})'
=\lambda^{(n_j-n_N)/p}\cdot (h'\circ f^{n_N})\cdot(f^{n_N})'\quad\text{on }U,
\end{gather*}
so that evaluating them at $z=z_0$, also by $h'(v_0)\neq 0$, we have
\begin{gather*}
(f^{n_j})'(z_0)=\lambda^{(n_j-n_N)/p}\cdot(f^{n_N})'(z_0)\quad\text{and making }j\to\infty,\\
g'(z_0)=a=\lambda_0\cdot(f^{n_N})'(z_0)
\end{gather*}
(here $m=1$). Hence for any $j\ge N$, we have
\begin{gather*}
\bigl(\lambda^{(n_j-n_N)/p}-\lambda_0\bigr)(f^{n_N})'(z_0)
=(f^{n_j})'(z_0)-a.
\end{gather*}
On the other hand, by \eqref{eq:Hartogs} (here $m=1$) and \cite[(3.8)]{OkuyamaDerivatives},
we have
\begin{gather*}
 \limsup_{j\to\infty}\frac{\log|(f^{n_j})'(z_0)-a|}{d^{n_j}-1}<0
\quad\text{and}\quad
 \lim_{j\to\infty}\frac{\log|\lambda^{(n_j-n_N)/p}-\lambda_0|}{d^{n_j}-1}=0.
\end{gather*}
Hence we have
\begin{gather}
 (f^{n_N})'(z_0)=0,\label{eq:vanishfirstiteration} 
\end{gather}
which with $a=\lambda_0\cdot (f^{n_N})'(z_0)$ yields $a=0$. \qed

\subsection*{(II)''} Now, assume that $m=2$
and that $g$ is non-constant.
For any $j\ge N$, differentiating both sides in \eqref{eq:functional} twice, 
by the chain rule, we have
\begin{gather*}
 (h'\circ f^{n_j})\cdot(f^{n_j})'
=\lambda^{(n_j-n_N)/p}\cdot (h'\circ f^{n_N})\cdot(f^{n_N})'\quad\text{and then}\\
(h''\circ f^{n_j})\cdot((f^{n_j})')^2+(h'\circ f^{n_j})\cdot(f^{n_j})''
=\lambda^{(n_j-n_N)/p}\cdot\bigl((h''\circ f^{n_N})\cdot((f^{n_N})')^2
+(h'\circ f^{n_N})\cdot(f^{n_N})''\bigr)
\end{gather*}
on $U$, so that evaluating them at $z=z_0$, also by $h'(v_0)\neq 0$, we have
\begin{gather}
(f^{n_j})'(z_0)=\lambda^{(n_j-n_N)/p}\cdot(f^{n_N})'(z_0)\quad\text{and}\label{eq:derivative}\\
h''(v_0)((f^{n_j})'(z_0))^2+h'(v_0)(f^{n_j})''(z_0)
=\lambda^{(n_j-n_N)/p}\cdot\bigl(h''(v_0)\cdot((f^{n_N})'(z_0))^2
+h'(v_0)(f^{n_N})''(z_0)\bigr),\label{eq:secondorderderivative}
\end{gather}
and in turn making $j\to\infty$, 
\begin{gather}
g'(z_0)=\lambda_0\cdot(f^{n_N})'(z_0)\quad\text{and}\label{eq:derivativelimit}\\
h''(v_0)(g'(z_0))^2+h'(v_0)a
=\lambda_0\cdot\bigl(h''(v_0)((f^{n_N})'(z_0))^2+h'(v_0)(f^{n_N})''(z_0)\bigr)\label{eq:secondderivativelimit}
\end{gather}
(here $m=2$ so $a=g''(z_0)$).
Hence for any $j\ge N$, subtracting \eqref{eq:secondderivativelimit} from 
\eqref{eq:secondorderderivative} and then eliminating $(f^{n_j})'(z_0)$
and $g'(z_0)$ by \eqref{eq:derivative} and \eqref{eq:derivativelimit},
the above four equalities yield
\begin{multline*}
 h''(v_0)\cdot\bigl((\lambda^{(n_j-n_N)/p})^2-\lambda_0^2\bigr)\bigl((f^{n_N})'(z_0)\bigr)^2
-h'(v_0)\bigl((f^{n_j})''(z_0)-a\bigr)\\
=\bigl(\lambda^{(n_j-n_N)/p}-\lambda_0\bigr)
\cdot\bigl(h''(v_0)\cdot((f^{n_N})'(z_0))^2+h'(v_0)\cdot(f^{n_N})''(z_0)\bigr),
\end{multline*}
which is rewritten as
\begin{multline}
 \frac{(f^{n_j})''(z_0)-a}{\lambda^{(n_j-n_N)/p}-\lambda_0}
=\frac{\bigl(\lambda^{(n_j-n_N)/p}+\lambda_0-1\bigr)h''(v_0)((f^{n_N})'(z_0))^2-h'(v_0)\cdot(f^{n_N})''(z_0)}{h'(v_0)}\\
=(\lambda^{(n_j-n_N)/p}-\lambda_0)\cdot\frac{h''(v_0)((f^{n_N})'(z_0))^2}{h'(v_0)}
+\frac{(2\lambda_0-1)h''(v_0)((f^{n_N})'(z_0))^2-h'(v_0)\cdot(f^{n_N})''(z_0)}{h'(v_0)}.
\label{eq:first}
\end{multline}
On the other hand, by \eqref{eq:Hartogs} (here $m=2$) and \cite[(3.8)]{OkuyamaDerivatives},
we have
\begin{gather}
 \limsup_{j\to\infty}\frac{\log|(f^{n_j})''(z_0)-a|}{d^{n_j}-2}<0
\quad\text{and}\quad
 \lim_{j\to\infty}\frac{\log|\lambda^{(n_j-n_N)/p}-\lambda_0|}{d^{n_j}-2}=0.\label{eq:apriorisecond}
\end{gather}

Hence making $j\to\infty$ in \eqref{eq:first}, we must have 
\begin{gather}
 (2\lambda_0-1)h''(v_0)((f^{n_N})'(z_0))^2-h'(v_0)\cdot(f^{n_N})''(z_0)=0,\label{eq:vanish} 
\end{gather}
which with \eqref{eq:first} in turn yields
\begin{gather}
\frac{(f^{n_j})''(z_0)-a}{(\lambda^{(n_j-n_N)/p}-\lambda_0)^2}
=\frac{(f^{n_N})''(z_0)}{2\lambda_0-1}\tag{\ref{eq:first}$'$}\label{eq:second}
\end{gather}
for any $j\ge N$.
Then by \eqref{eq:apriorisecond} again, 
from \eqref{eq:second},
we have
\begin{gather}
 (f^{n_N})''(z_0)=0,\label{eq:vanishseconditeration} 
\end{gather}
which with \eqref{eq:vanish} and \eqref{eq:derivativelimit} yields
\begin{gather}
h''(v_0)((f^{n_N})'(z_0))^2=0\quad\text{and}\quad
0=\lambda_0^2\cdot h''(v_0)((f^{n_N})'(z_0))^2=h''(v_0)(g'(z_0))^2 .\label{eq:vanishfirstiterationlimit} 
\end{gather}
Consequently, by \eqref{eq:secondderivativelimit}, 
\eqref{eq:vanishseconditeration},
\eqref{eq:vanishfirstiterationlimit},
and $h'(v_0)\neq 0$, we have $a=0$. \qed

\section{Proofs of Theorems \ref{th:derivative} and \ref{th:adelicequidist}}\label{sec:overnotC}

\subsection{Non-archimedean dynamics of polynomials of degree $>1$}\label{sec:backgroundnonarchi}
Let $K$ be an algebraically closed field that is complete with respect
to a non-trivial and non-archimedean absolute value $|\cdot|$.
The Berkovich projective line $\sP^1=\sP^1(K)$ is a compact augmentation
of the {\em classical} projective line $\bP^1=\bP^1(K)$ and is also
locally compact, Hausdorff, and uniquely arcwise connected. 
Let us see more details. As a set, the Berkovich affine line
$\sA^1=\sA^1(K)$
is the set of all multiplicative seminorms 
$K[z]$ which restricts to $|\cdot|$ on $K$. We write an element of $\sA^1$
like $\cS$ and denote it by $[\cdot]_{\cS}$ as a multiplicative
seminorm on $K[z]$. A {\em $K$-closed disk} is a subset in $K$
written as $B(a,r):=\{z\in K:|z-a|\le r\}$ for some $a\in K$ and $r\ge 0$;
by the strong triangle inequality, for any $b\in B(a,r)$,
we have $B(b,r)=B(a,r)$, and for any two $K$-closed disks $B,B'$
having non-empty intersection,
we have either $B\subset B'$ or $B\supset B'$.
By Berkovich's representation \cite{Berkovichbook},
any element $\cS\in\sA^1$ is induced by a non-increasing and nesting sequence 
$(B_n)$ of $K$-closed disks 
in that
\begin{gather}
[\phi]_{\cS}=\inf_{n\in\bN}\sup_{z\in B_n}|\phi(z)|\quad\text{for any }\phi\in K[z].\label{eq:seminorm} 
\end{gather}
In particular, each point $a\in K$ is regarded as an element of $\sA^1$
induced by the (constant sequence of the) $K$-closed disk $B(a,0)=\{a\}$, and
more generally, each $K$-closed disk $B$ is regarded as an element of $\sA^1$
induced by (the constant sequence of) $B$. 
In particular, $K$ is regarded as a subset of $\sA^1$.
The relative topology of $\sA^1$ is the weakest topology such that 
for any $\phi\in K[z]$, $\sA^1\ni\cS\mapsto[\phi]_{\cS}\in\bR_{\ge 0}$ is continuous,
and then $\sA^1$ is a locally compact, uniquely arcwise connected,
Hausdorff topological space. 
The action on $K$ of a polynomial $h\in K[z]$ 
continuously extends to $\sA^1$ as
\begin{gather}
 [\phi]_{h(\cS)}=[\phi\circ h]_{\cS}\quad\text{for every }\cS\in\sA^1,\label{eq:action}
\end{gather}
preserving $K$ and $\sA^1\setminus K$ if in addition $\deg h>0$.

As a set, $\sP^1$ is nothing but $\sA^1\cup\{\infty\}$,
regarding $\bP^1$ as $K\cup\{\infty\}$, and as a topological space,
$\sP^1$ is identified with the one-point compactification of $\sA^1$.
An ordering $\le_\infty$ on $\sA^1$ is defined
so that for any $\cS,\cS'\in\sA^1$, $\cS\le_\infty\cS'$
if and only if $[\cdot]_{\cS}\le_\infty[\cdot]_{\cS'}$ on $K[z]$,
and this $\le_\infty$ extends to the ordering on $\sP^1$ 
so that $\cS\le_\infty\infty$ for every $\cS\in\sP^1$.
For any $\cS,\cS'\in\sP^1$, if $\cS\le_{\infty}\cS'$, then set 
$[\cS,\cS']=[\cS',\cS]:=\{\cS''\in\sP^1:\cS\le_\infty\cS''\le_\infty\cS'\}$,
and in general, we have 
$[\cS,\infty]\cap[\cS',\infty]=[\cS\wedge_\infty\cS',\infty]$, 
for some (unique) $\cS\wedge_\infty\cS'\in\sP^1$, and then set 
$[\cS,\cS']:=[\cS,\cS\wedge_\infty\cS']\cup[\cS\wedge_\infty\cS',\cS']$.
These {\em closed intervals} $[\cS,\cS']\subset\sP^1$ make $\sP^1$ 
an ``$\bR$-''tree in the sense of Jonsson \cite[Definition 2.2]{Jonsson15}.
For any $\cS\in\sP^1$, the equivalence class 
$T_{\cS}\sP^1:=(\sP^1\setminus\{\cS\})/\sim$ is defined so that
for any $\cS',\cS''\in\sP^1\setminus\{\cS\}$, 
$\cS'\sim\cS''$ if $[\cS,\cS']\cap[\cS,\cS'']=[\cS,\cS'\wedge_{\cS}\cS'']$ 
for some (unique) $\cS'\wedge_{\cS}\cS''\in\sP^1\setminus\{\cS\}$. 
An element $\Vv$ of $T_{\cS}\sP^1$ is called a {\itshape direction} of $\sP^1$
at $\cS$, which is denoted by $U(\Vv)$ as a
subset in $\sP^1\setminus\{\cS\}$ and, if $\cS'\in U(\Vv)$,
also by $\overrightarrow{\cS\cS'}$.
A non-empty subset in $\sP^1$ is called a {\em simple domain}
if it is the intersection of some finitely many elements of 
$\{U(\Vv):\cS\in\sP^1,\Vv\in T_{\cS}\sP^1,\#T_{\cS}\sP^1>1\}$.
The topology of $\sP^1$ has an open basis consisting of all 
simple domains in $\sP^1$.

The point $[\cdot]_{\cO_K}$ in $\sP^1$, where $\cO_K:=\{z\in K:|z|\le 1\}$
is the ring of $K$-integers,
is called the {\em Gauss} or {\em canonical} point in $\sP^1$ and
is denoted by $\cS_{\can}$. 
Let us denote the continuous extension of $|\cdot|$ to 
$\sA^1$ by the same $|\cdot|$ for simplicity. More generally, 
let $|\cS-\cS'|$ be the {\em Hsia kernel} on 
$\sA^1$, which is
the upper semicontinuous and separately continuous extension 
to $\sA^1\times\sA^1$ of the function $|z-w|$ on $K\times K$ 
(although $\cS-\cS'$ itself is undefined unless $\cS,\cS'\in K$), and
then the function $\log|\cS-\cS'|-\log\max\{1,|\cS|\}-\log\max\{1,|\cS'|\}$
on $\sP^1\times\sP^1$ is the {\em generalized} Hsia kernel on $\sP^1$
{\em with respect to $\cS_{\can}$}, which is the upper semicontinuous
and separately continuous extension to $\sP^1\times\sP^1$
of the (normalized) chordal metric on $\bP^1$ (\cite[\S4.4]{BR10}).

The function
$\log\max\{1,|\cdot|\}$ on $\sA^1=\sP^1\setminus\{\infty\}$ extends superharmonically near $\infty$ so that
\begin{gather*}
 \Delta\log\max\{1,|\cdot|\}=\delta_{\cS_{\can}}-\delta_\infty\quad\text{on }\sP^1.
\end{gather*}
Here, the Laplacian on $\sP^1$ is denoted by $\Delta:=\Delta_{\sP^1}$
(in \cite{BR10} the opposite sign convention on $\Delta$ is adopted).
A function $g:\sP^1\to\bR\cup\{\pm\infty\}$ is 
said to be $\delta_{\cS_{\can}}$-subharmonic if 
\begin{gather}
 \mu_g:=\Delta g+\delta_{\cS_{\can}}
\end{gather}
is a probability Radon measure on $\sP^1$; for example, 
$-\log\max\{1,|\cdot|\}$ is a $\delta_{\cS_{\can}}$-subharmonic function 
on $\sP^1$. If in addition $g$ is an 
$\bR$-valued continuous function on $\sP^1$, then the function
\begin{gather}
 \cS\mapsto\int_{\sP^1}\bigl(\log|\cS-\cS'|-(g(\cS)+\log\max\{1,|\cS|\})
-(g(\cS')+\log\max\{1,|\cS'|\})\bigr)\mu_g(\cS')
%\equiv\int_{\sP^1\times\sP^1}\bigl(\log|\cS-\cS'|-(g(\cS)+\log\max\{1,|\cS|\})
%-(g(\cS')+\log\max\{1,|\cS'|\})\bigr)(\mu_g\times\mu_g)(\cS,\cS')
\label{eq:Frostman}
\end{gather}
is constant on $\sP^1$ 
(see \cite[Proposition 8.70]{BR10}).

 The continuous
action on $\bP^1$ of a rational function
$h\in K(z)$ canonically extends
to $\sP^1$. If in addition $h$ is non-constant, then the action
of $h$ on $\sP^1$ preserves both $\bP^1$ and $\sP^1\setminus\bP^1$
and is open and surjective. The local degree function $w\mapsto \deg_{w}h$
on $\bP^1$ also canonically extends to an upper semi-continuous function
on $\sP^1$, satisfying $\sum_{\cS'\in h^{-1}(\cS)}\deg_{\cS'}h=\deg h$
for each $\cS\in\sP^1$.
In particular, the action of $h$ on $\sP^1$ induces the {\em pull-back} action
on the space of Radon measures on $\sP^1$ so that, letting $\delta_{\cS}$
be the Dirac measure on $\sP^1$ at each $\cS\in\sP^1$, 
$h^*\delta_{\cS}=\sum_{\cS'\in h^{-1}(\cS)}(\deg_{\cS'}h)\delta_{\cS'}$ on $\sP^1$.

Let $f\in K[z]$ be a polynomial of degree $d>1$. 
The {\em Berkovich} filled-in Julia set of $f$ is
\begin{gather*}
\sK(f):=\Bigl\{\cS\in\sA^1:\limsup_{n\to\infty}|f^n(\cS)|<\infty\Bigr\},
\end{gather*}
which is a compact subset in $\sA^1$,
and the {\em escape rate function} of $f$ on $\sA^1$ is the limit 
$g_f:=\lim_{n\to\infty}(\log\max\{1,|f^n|\})/d^n$ on $\sA^1$.
The difference $g_f-(\log\max\{1,|f^n|\})/d^n$ on $\sA^1$
is harmonic and bounded on a neighborhood of $\infty$, 
so it extends harmonically across $\infty$ (see e.g.\ \cite[\S 7]{BR10}), and
we have the estimate
\begin{gather}
g_f-\frac{\log\max\{1,|f^n|\}}{d^n}=O(d^{-n})\quad\text{as }n\to\infty\label{eq:Greeniteratenonarch}
\end{gather}
on $\sP^1$ uniformly. 
The function $g_f$ is continuous, subharmonic, 
and $\ge 0$ on $\sA^1$, it is harmonic and $>0$ on $\sA^1\setminus\sK(f)$,
and it is $=0$ on $\sK(f)$.
The {\em equilibrium} (or {\em canonical}) measure of $f$
is the probability Radon measure
\begin{gather*}
\mu_f:=\Delta g_f+\delta_\infty\quad\text{on }\sP^1,
\end{gather*}
which is supported exactly by $\partial\sK(f)$.
The {\em Berkovich} superattractive basin 
\begin{gather*}
\sI_\infty(f):=\Bigl\{z\in\sP^1:\lim_{n\to\infty}f^n(z)=\infty\Bigr\}
\end{gather*}
of $f$ associated to the superattracting fixed point $\infty$ of $f$ 
is a domain in $\sP^1$ containing $\infty$, and
coincides with $\sP^1\setminus\sK(f)$. 
Let $C(f)$ be the (classical) critical set of $f$, 
which consists of $\infty$ and all the (at most $d-1$) zeros of $f'$ on $K$.
Then
$\bigcup_{n\in\bN\cup\{0\}}f^{-n}(C(f)\setminus\{\infty\})$ is bounded in $K$. 

The {\em Berkovich Julia set} 
of $f$ is defined as
\begin{gather*}
 \sJ(f):=\supp\mu_f=\partial\sK(f).
\end{gather*}
The {\em Berkovich Fatou set} $\sF(f)$ of $f$ is defined by 
$\sP^1\setminus\sJ(f)$,
and a component of $\sF(f)$ is called a {\em Berkovich Fatou 
component} of $f$. Both $\sJ(f)$ and $\sF(f)$ are totally invariant
under $f$ and any Berkovich 
Fatou component of $f$ is either $\sI_\infty(f)$ or a component 
of the interior of $\sK(f)$. 

Set $c_d:=\lim_{K\ni z\to\infty}f(z)/z^d\in K^*=K\setminus\{0\}$.
Since $g_f-\log\max\{1,|\cdot|\}$ is an $\bR$-valued continuous and $\delta_{\cS_{\can}}$-subharmonic function on $\sP^1$ and satisfies
$\Delta(g_f-\log\max\{1,|\cdot|\})+\delta_{\cS_{\can}}=\mu_f$ on $\sP^1$, 
by \eqref{eq:Frostman}, the function
$\cS\mapsto\int_{\sP^1}\log|\cS-\cS'|\mu_f(\cS')-g_f(\cS)$
is constant on $\sP^1$. This with \eqref{eq:Greeniteratenonarch} yields
the identity
\begin{gather}
\int_{\sP^1}\log|\cS-\cS'|\mu_f(\cS')
\equiv g_f(\cS)-\frac{\log|c_d|}{d-1}
\bigl(=\log|\cS|+O(1/|\cS|)\text{ as }\cS\to\infty\bigr)\label{eq:potentialgreen}
\quad\text{on }\sP^1.
\end{gather}
For more details on the harmonic analysis and dynamics on $\sP^1$, see \cite{BR10,FRL}.

\subsection{Arithmetic dynamics of polynomials of degree $>1$}\label{sec:height}
Let $k$ be a product formula field as in Subsection \ref{sec:adelic}.
Let $f\in k[z]$ be a polynomial of degree $d>1$. For each $v\in M_k$,
we obtain $g_{f,v}$ and $\mu_{f,v}$ on $\sP^1(\bC_v)$ from the action of $f$
on $\bP^1(\bC_v)$. 
Writing $f(z)$ as $\sum_{j=0}^dc_jz^j\in k[z]$, so $c_d\in k^*$, 
there is
a finite set $E_f$ containing all the infinite places of $k$ such that
for every $v\in M_k\setminus E_f$, 
$|c_d|_v=1, |c_0|_v,\ldots,|c_{d-1}|_v\le 1$,
and moreover, $g_{f,v}=\log\max\{1,|\cdot|_v\}$ and $\mu_{f,v}=\delta_{\cS_{\can,v}}$
on $\sP^1(\bC_v)$.

Recall that an embedding of $\overline{k}$ to $\bC_v$ is fixed
for each $v\in M_k$. The Call-Silverman {\em $f$-canonical height} of an effective 
$k$-divisor $\cZ$ on $\bP^1(\overline{k})$ supported by $\overline{k}$ is
\begin{align}
0\le\hat{h}_f(\cZ):=&\sum_{v\in M_k}N_v\frac{\sum_{z\in\overline{k}:p(z)=0}(\deg_z p)g_{f,v}(z)}{\deg p}\label{eq:height} \\
\notag=&h_{\operatorname{nv}}(\cZ)
+\sum_{v\in E_f}N_v\frac{\sum_{z\in\overline{k}:p(z)=0}(\deg_z p)\bigl(g_{f,v}(z)-\log\max\{1,|z|_v\}\bigr)}{\deg p},
\end{align}
where $p\in k[z]$ is a representative of $\cZ$ and the {\em naive} height
\begin{gather*}
 h_{\operatorname{nv}}(\cZ):=\sum_{v\in M_k}N_v\frac{\sum_{z\in\overline{k}:p(z)=0}(\deg_z p)\log\max\{1,|z|_v\}}{\deg p}
\end{gather*}
of $\cZ$ is in fact a {\em finite} sum
by a standard argument involving the ramification theory
of valuations (or \cite[Lemma 2.3]{OkuDivisor}).
For every $v\in M_k$, setting $a_p:=p^{(\deg p)}/(\deg p)!\in k^*$, 
we have $\log|p(\cdot)|_v=\sum_{z\in\overline{k}:p(z)=0}
(\deg p)\log|\cdot-z|_v+\log|a_p|_v$
on $\sA^1(\bC_v)$, integrating both sides in which 
against $\mu_{f,v}$ over $\sP^1(\bC_v)$, also by \eqref{eq:potentialgreen},
we have
\begin{align*}
 \int_{\sP^1(\bC_v)}\log|p|_v\mu_{f,v}
=&\sum_{z\in\overline{k}:p(z)=0}(\deg_z p)\int_{\sP^1(\bC_v)}\log|z-\cS'|_v\mu_{f,v}(\cS')+\log|a_p|_v\\
=&\sum_{z\in\overline{k}:p(z)=0}(\deg_z p)g_{f,v}(z)
-(\deg p)\cdot\frac{\log|c_d|_v}{d-1}+\log|a_p|_v.
\end{align*}
Consequently, also by the product formula property of $k$, 
the defining equality \eqref{eq:height} of $\hat{h}_f(\cZ)$
is rewritten as the {\em Mahler-type formula}
\begin{gather}
 \hat{h}_f(\cZ)=\sum_{v\in M_k}N_v\frac{\int_{\sP^1(\bC_v)}\log|p|_v\mu_{f,v}}{\deg p}\tag{\ref{eq:height}$'$}\label{eq:Mahler}
\end{gather}
(cf.\ \cite[(1.1)]{OkuDivisor}).
For more details on canonical heights on $\sP^1$, 
see \cite{baker-hsia, BR06, FRL, ACL}. For the treatment of effective divisors
rather than Galois conjugacy classes, which are effective divisors
represented by {\em irreducible} polynomials, see \cite{OkuDivisor}.

\subsection{Proofs of Theorems \ref{th:derivative} and \ref{th:adelicequidist}}
\label{sec:adelic2}

Let $K$ be an algebraically closed field of characteristic $0$
that is complete with respect
to a non-trivial and non-archimedean absolute value $|\cdot|$. 
Let $f\in K[z]$ be a polynomial of degree $d>1$, and fix $m\in\bN$.

The following is a non-archimedean counterpart to Lemma \ref{th:basin}.

\begin{lemma}\label{th:basinnonarchi}
We have
\begin{gather}
 (f^n)^{(m)}=\bigl((e^{O(1)}\cdot d^n)^m+O(d^{(m-1)n})\bigr)\cdot f^n\quad\text{as }n\to\infty\tag{\ref{eq:higherorderderivative}$'$}\label{eq:higherderivativenonarch}
\end{gather}
on $\sI_\infty(f)\setminus\bigcup_{n\in\bN\cup\{0\}}f^{-n}(C(f))$
locally uniformly. Moreover, for every $a\in K$,
the family $\bigl((\log|(f^n)^{(m)}-a|)/(d^n-m)-\log\max\{1,|\cdot|\}\bigr)_n$ 
of $\delta_{\cS_{\can}}$-subharmonic functions on $\sP^1$
is locally uniformly bounded from above on $\sP^1$ and
\begin{gather}
\lim_{n\to\infty}\biggl(\frac{\log|(f^n)^{(m)}-a|}{d^n-m}-g_f\biggr)=0\tag{\ref{eq:aroundinfty}$'$}\label{eq:aroundinftynonarchi}
\end{gather}
on $\sI_\infty(f)\setminus\bigcup_{n\in\bN\cup\{0\}}f^{-n}(C(f))$ locally uniformly.
\end{lemma}

\begin{proof}
Fixing $r\gg 1$, there is a $($rigid$)$ {\em biholomorphism}
$w=\psi(z)$ from $\bP^1\setminus\{g_f\le r\}$ to $\bP^1\setminus\{|w|\le e^r\}$,
which is called a (non-archimedean) B\"ottcher coordinate
near $\infty$ associated to $f$, such that 
$\psi(f(z))=\psi(z)^d$ on $\bP^1\setminus\{g_f\le r\}$
(see Rivera-Letelier \cite[the proof of Proposition 3.3(ii)]{Juan03}).
Then 
$\psi(\infty)=\infty$ and $\psi'\neq 0$ on $\bP^1\setminus\{g_f\le r\}$.
By a computation similar to that in the proof of Lemma \ref{th:basin},
we have
\begin{gather}
\frac{(f^n)'}{f^n}(z)
=d^n\cdot\bigl(1+O(\psi(z)^{-d^n})\bigr)\frac{\psi'}{\psi}(z)\quad\text{as }n\to\infty\tag{\ref{eq:nearinfty}$'$}\label{eq:lognonarchiaroundinfty} 
\end{gather}
on $K\setminus\{g_f\le r\}$ uniformly.

For any simple domain $D\Subset I_\infty(f)\cap\sA^1$ 
and any $M\in\bN\cup\{0\}$ so large that
$f^M(D)\subset\sP^1\setminus\{g_f\le r\}$, 
from \eqref{eq:lognonarchiaroundinfty}, we also have
\begin{gather*}
\frac{(f^n)'}{f^n}
%=\frac{((f^{n-M})'\circ f^M)\cdot(f^M)'}{f^{n-M}\circ f^M}
=d^{n-M}\cdot\frac{\psi'}{\psi}\circ f^M\cdot(f^M)'+o(1) 
\quad\text{as }n\to\infty
\end{gather*}
on $D\cap\bP^1$ uniformly. Fix now $m\in\bN$. Then
noting that, by the definition of a simple domain,
there is $0<\epsilon\ll 1$ such that $B(z,\epsilon)\subset D\cap\bP^1$ 
for any $z\in D\cap\bP^1$, 
an induction which is similar to that in the proof of Lemma \ref{th:basin} and
involves the almost straightforward (non-archimedean)
Cauchy's estimate for (rigid) analytic functions on those disks
$B(z,\epsilon)$ yields
\begin{gather}
\frac{(f^n)^{(m)}}{f^n}
=\biggl(d^{n-M}\cdot\frac{\psi'}{\psi}\circ f^M\cdot(f^M)'\biggr)^m
+O(d^{(m-1)n})\quad\text{as }n\to\infty\tag{\ref{eq:logarithmic}$'$}\label{eq:logarithmicnonarch}
\end{gather}
on $D\cap\bP^1$ uniformly. 
If in addition 
$D\Subset I_\infty(f)\setminus\bigcup_{n\in\bN\cup\{0\}}f^{-n}(C(f))$, 
so $\inf_{D}|\frac{\psi'}{\psi}\circ f^M\cdot(f^M)'|>0$,
then this \eqref{eq:logarithmicnonarch}
yields the asymptotic estimate 
\eqref{eq:higherderivativenonarch} on $D\cap\bP^1$ uniformly, 
and in turn on $D$ uniformly by 
the continuity of $|(f^n)^{(m)}/f^n|$ on $D$ and
the density of $\bP^1$ in $\sP^1$. 

Also fix $a\in K$. The locally uniform convergence \eqref{eq:aroundinftynonarchi}
on $\sI_\infty(f)\setminus\bigcup_{n\in\bN\cup\{0\}}f^{-n}(C(f))$ 
follows from the estimate \eqref{eq:higherderivativenonarch}. 
In particular, for $R\gg 1$,
letting $\cS_R\in[0,\infty]\setminus\bP^1$ be the point in $\sP^1\setminus\bP^1$ 
induced by 
the (constant sequence of the) $K$-closed disk $B(0,R)$ (so 
$T_{\cS_R}\sP^1\supset\{\overrightarrow{\cS_R0},\overrightarrow{\cS_R\infty}\}$),
we have the convergence \eqref{eq:aroundinftynonarchi} at $\cS=\cS_R$,
and in turn, by the maximum principle for subharmonic functions 
(cf.\ \cite[Proposition 8.14]{BR10}), the family 
$\bigl(\log|(f^n)^{(m)}-a|/(d^n-m)\bigr)_n$ 
is uniformly bounded from above 
on $U(\overrightarrow{\cS_R0})$ (whose boundary is $\{\cS_R\}$). Similarly, for $R\gg 1$, noting that
$\log\bigl|(f^n)^{(m)}/f^n\bigr|$ is a subharmonic function 
on $U(\overrightarrow{\cS_R\infty})$
(whose boundary is $\{\cS_R\}$), 
by the maximum principle for subharmonic functions
(and \eqref{eq:Greeniteratenonarch}), we have
\begin{align*}
 \frac{\log|(f^n)^{(m)}|}{d^n-m}-\log\max\{1,|\cdot|\}
\le&\biggl(\frac{\log|f^n|}{d^n-m}+O(nd^{-n})\biggr)-\log\max\{1,|\cdot|\}\\
=&g_f-\log\max\{1,|\cdot|\}+O(nd^{-n})=O(nd^{-n})\quad\text{as }n\to\infty
\end{align*}
on $U(\overrightarrow{\cS_R\infty})$ uniformly.
Hence
the family $\bigl((\log|(f^n)^{(m)}-a|)/(d^n-m)-\log\max\{1,|\cdot|\}\bigr)_n$ 
is locally uniformly bounded from above on $\sP^1$.
\end{proof}

Fix also $a\in K$. 
By the second and the last assertions in Lemma \ref{th:basinnonarchi},
a {\em compactness principle} for a family of $\delta_{\cS_{\can}}$-subharmonic
functions on $\sP^1$
(cf.\ \cite[Proposition 2.18]{FRL}, \cite[Proposition 8.57]{BR10}) yields 
a sequence $(n_j)$ in $\bN$ tending to $\infty$ as $j\to\infty$ and
a function $\phi:\sP^1\to\bR\cup\{-\infty\}$ such that
\begin{multline*}
 \phi=
\lim_{j\to\infty}\biggl(\frac{\log|(f^{n_j})^{(m)}-a|}{d^{n_j}-m}-g_f\biggr)\\
\Biggl(=\lim_{j\to\infty}\biggl(\Bigl(\frac{\log|(f^{n_j})^{(m)}-a|}{d^{n_j}-m}-\log\max\{1,|\cdot|\}\Bigr)-(g_f-\log\max\{1,|\cdot|\})\biggr)\Biggr)\quad\text{on }\sP^1\setminus\bP^1
\end{multline*}
and that 
\begin{gather*}
 \Delta\phi+\mu_f\bigl(=\Delta(\phi+g_f-\log\max\{1,|\cdot|\})+\delta_{\cS_{\can}}
=\Delta(\phi+g_f)+\delta_\infty\bigr) 
\end{gather*}
is a {\em probability} Radon measure on $\sP^1$.
By \eqref{eq:aroundinftynonarchi}, 
we have $\phi\equiv 0$ on $\sI_\infty(f)\setminus\bP^1$, and in turn
$\phi\equiv 0$ on $\sI_\infty(f)$
by the subharmonicity of $\phi=(\phi+g_f)-g_f$ on $\sI_\infty(f)\cap\sA^1$
and the maximum principle for subharmonic functions
(cf.\ \cite[Proposition 8.14]{BR10}).
Then also by $\sI_\infty(f)=\{g_f>0\}$,
the subharmonicity of $\phi+g_f$ on $\sA^1$, and
the maximum principle for subharmonic functions again,
we have $\phi+g_f\le\max_{\{g_f=\epsilon\}}(\phi+g_f)=0+\epsilon=\epsilon$ 
on $\sK(f)=\{g_f=0\}\subset\{g_f<\epsilon\}$ for every $\epsilon>0$,
so that $\phi(=\phi+g_f)\le 0$ on $\sK(f)$.
Then we also have 
$\phi\equiv 0$ on $\sJ(f)=\partial\sK(f)$. 

Let us see that
\begin{gather}
 \limsup_{n\to\infty}\frac{\int_{\sP^1}\log|(f^n)^{(m)}-a|\mu_f}{d^{n}-m}\le 0, 
\label{eq:localheight}
\end{gather}
which will be used in the proof of Theorem \ref{th:adelicequidist};
indeed, 
\begin{gather*}
\limsup_{j\to\infty}\frac{\int_{\sP^1}\log|(f^{n_j})^{(m)}-a|\mu_f}{d^{n_j}-m}
\le\limsup_{j\to\infty}\sup_{\sJ(f)}\frac{\log|(f^{n_j})^{(m)}-a|}{d^{n_j}-m}
\le\sup_{\sJ(f)}(\phi+g_f)=0,
\end{gather*}
where the first inequality is by $\supp\mu_f=:\sJ(f)$, and
the second one is by a version of Hartogs's lemma for a sequence 
of $\delta_{\cS_{\can}}$-subharmonic functions on $\sP^1$
(cf.\ \cite[Proposition 2.18]{FRL}, \cite[Proposition 8.57]{BR10}).

\begin{proof}[Proof of Theorem $\ref{th:derivative}$]
 Let us continue the above argument.
Suppose now that the open subset $\{\phi<0\}$ is non-empty. 
Then since $\phi\equiv 0$ on $\sI_{\infty}(f)$,
there is a Berkovich Fatou component $U$ of $f$ other than $\sI_\infty(f)$
(so $U\Subset\sA^1$)
such that $U\cap\{\phi<0\}\neq\emptyset$, 
and $\partial U$ is a singleton, say $\{\cS_0\}$,
in $\sP^1\setminus\bP^1$ (see \cite[Lemma 2.1]{Okuapripri}).
Moreover, $\phi\equiv 0$ on $\partial U\subset\sJ(f)$. 

Assume in addition that $f$ has no potentially good reductions. Then 
in particular, $\mu_f(\partial U)(=\mu_f(\{\cS_0\}))=0$. 
Now setting
\begin{gather*}
 \psi:=\begin{cases}
	\phi & \text{on }U\\
	0 & \text{on }\sP^1\setminus U
       \end{cases}:\sP^1\to\bR_{\le 0}\cup\{-\infty\}
\end{gather*}
and checking that the function $\psi+g_f$ is subharmonic on $\sA^1$,
we conclude $\psi\equiv 0$ on $\sP^1$ 
by an argument similar to that in 
\cite[Proof of Theorem 1]{Okuapripri}
involving a Bedford-Taylor-type {\em domination principle}
(see \cite[\S 4]{Okuapripri}).
This contradicts $U\cap\{\phi<0\}\neq\emptyset$. 

Hence $\phi\equiv 0$ on $\sP^1$
under the no potentially good reductions condition on $f$.
Then \eqref{eq:derivatives} follows from the equality
\begin{gather*}
 \Delta\Bigl(\frac{\log|(f^n)^{(m)}-a|}{d^n-1}-g_f\Bigr)=
\frac{\bigl((f^n)^{(m)}\bigr)^*\delta_a}{d^n-1}-\mu_f\quad\text{on }\sP^1
\end{gather*}
and a continuity of the Laplacian $\Delta$.
\end{proof}

\begin{proof}[Proof of Theorem $\ref{th:adelicequidist}$]
Let $k$ be a product formula field of characteristic $0$ 
and let $f\in k[z]$ be a polynomial of degree $d>1$. Recall that,
writing $f(z)$ as $\sum_{j=0}^dc_jz^j\in k[z]$, so $c_d\in k^*$, there is
a finite subset $E_f$ in $M_k$ containing all the infinite places of $k$ such that
for every $v\in M_k\setminus E_f$, 
\begin{gather*}
 |c_d|_v=1,\quad |c_0|_v,|c_1|_v,\ldots,|c_{d-1}|_v\le 1 
\end{gather*}
and moreover,
$g_{f,v}=\log\max\{1,|\cdot|_v\}$ and $\mu_{f,v}=\delta_{\cS_{\can,v}}$
on $\sP^1(\bC_v)$, regarding $f\in\bC_v[z]$. 

Fix $m\in\bN$ and $a\in k$. For every $n\in\bN$,
$(f^n)^{(m)}\in(\bZ[c_0,\ldots,c_d])[z]$ by induction. By the product
formula property of $k$, there is
an at most finite (and possibly empty) subset $E_a$ in $M_k$ 
such that for every $v\in M_k\setminus E_a$, $|a|_v\in\{0,1\}$. 
Then for every $n\in\bN$ and every $v\in M_k\setminus(E_f\cup E_a)$,
%also using the strong triangle inequality, 
we have 
%both
\begin{multline*}
\int_{\sP^1(\bC_v)}\log|(f^n)^{(m)}-a|_v\mu_{f,v}
\le \int_{\sP^1(\bC_v)}\log \max\{|(f^n)^{(m)}|_v,|a|_v\}\delta_{\cS_{\can,v}}\\
=\log\max\Bigl\{\sup_{z\in\cO_{\bC_v}}|(f^n)^{(m)}(z)|_v,|a|_v\Bigr\}
\le\log\max\{|c_0|_v,\ldots,|c_d|_v,|a|_v\}=\log 1=0
\end{multline*}
(see \eqref{eq:seminorm} and \eqref{eq:action} for the first equality),
% and
%\begin{align*}
%  \int_{\sP^1(\bC_v)}\log\max\{1,|\cdot|_v\}\mu_{f,v}
%=&\int_{\sP^1(\bC_v)}\log\max\{1,|\cdot|_v\}\delta_{\cS_{\can,v}}\\
%=&\log\max\{1,|\cS_{\can,v}|_v\}
% =\log\max\Bigl\{1,\max_{z\in\cO_{\bC_v}}|z|_v\Bigr\}=\log 1=0,
%\end{align*}
which with the second assertions in Lemmas \ref{th:basinnonarchi} and \ref{th:basin}
(for finite and infinite $v\in M_k$, respectively) 
%imply 
implies that
\begin{gather*}
\sup_{v\in M_k}\sup_{n\in\bN}N_v
\frac{\int_{\sP^1(\bC_v)}\log|(f^n)^{(m)}-a|_v\mu_{f,v}}{d^n-m}
%\le\sup_{v\in E_f\cup E_a}\sup_{n\in\bN}N_v
%\frac{\int_{\sP^1(\bC_v)}\log|(f^n)^{(m)}-a|_v\mu_{f,v}}{d^n-m}
<\infty.
\end{gather*}
Now by the Mahler-type formula \eqref{eq:Mahler}, Fatou's lemma,
and \eqref{eq:localheight}, we have
\begin{gather*}
\limsup_{n\to\infty}\hat{h}_f([(f^n)^{(m)}=a])
%=\limsup_{n\to\infty}\sum_{v\in M_k}N_v\frac{\int_{\sP^1(\bC_v)}\log|(f^n)^{(m)}-a|_v\mu_{f,v}}{d^n-m}\\
\le\sum_{v\in M_k}\limsup_{n\to\infty}N_v\frac{\int_{\sP^1(\bC_v)}\log|(f^n)^{(m)}-a|_v\mu_{f,v}}{d^n-m}\le 0,
\end{gather*}
which with the non-negativity \eqref{eq:height} of $\hat{h}_f$ yields
the {\em small $(g_{f,v})_{v\in M_k}$-heights property} \eqref{eq:heightvanishing} 
of the sequence $([(f^n)^{(m)}=a])_n$ of effective $k$-divisors 
on $\bP^1(\overline{k})$. 

We note that $\deg[(f^n)^{(m)}=a]=d^n-m\to\infty$ as $n\to\infty$
and that, whenever $v\in M_k$ is infinite, we have $\bC_v\cong\bC$.
%and $\mu_{f,v}$ has no atoms on $\sP^1(\bC_v)\cong\bP^1(\bC)$. 
Suppose now that $k$ is a number field and 
that $a\in k^*$, and
choose an infinite place $v\in M_k$ of $k$.
Then from the equidistribution \eqref{eq:equidistderivative}
of $(((f^n)^{(m)})^*\delta_a/(d^n-m))_n$ towards $\mu_{f,v}$, 
which has no atoms, on $\sP^1(\bC_v)\cong\bP^1(\bC)$, 
we have
$\sup_{w\in\bP^1(\overline{k}):(f^n)^{(m)}(w)=a}\deg_w((f^n)^{(m)})
=o((\deg[(f^n)^{(m)}=a]))$ as $n\to\infty$, so in particular
the {\em small diagonal property} 
\begin{gather*}
\sum_{w\in\bP^1(\overline{k}):(f^n)^{(m)}(w)=a}\bigl(\deg_w((f^n)^{(m)})\bigr)^2
=o\bigl((\deg[(f^n)^{(m)}=a])^2\bigr)\quad\text{as }n\to\infty
\end{gather*}
of $([(f^n)^{(m)}=a])_n$. Now the uniform 
{\em asymptotically $(g_{f,v})_{v\in M_k}$-Fekete configuration property}
\eqref{eq:Fekete} 
of $([(f^n)^{(m)}=a])_n$ holds (see \cite[Theorem 1]{okuadelically}), 
so in particular the adelic equidistribution \eqref{eq:adelicequidist} holds.
\end{proof}

\section{Proof of Theorem \ref{th:eigenvalue_automorphism}}
\label{sec:henon}

Let us first show a slightly more general equidistribution 
statement \eqref{eq:normalized} under the following
normalization \eqref{eq:indeterminancies} below.
Let $f$ be a H\'enon-type polynomial automorphism of $\C^2$ of degree $d>1$ 
normalized as
\begin{gather}
I^+=\{[0:0:1]\}\quad\text{and}\quad I^-=\{[0:1:0]\}.\label{eq:indeterminancies} 
\end{gather}
Then the function
\begin{gather*}
 (z,w)\mapsto g^+(z,w)-\log\max\{1,|z|\}\quad\text{on }\bC^2
\end{gather*}
extends pluriharmonically to an open neighborhood
of $L_\infty\setminus  I^+$ in $\bP^2$ (\cite[Theorem 6.1]{DS}).
Moreover, for every $n\in\bN$, writing $f^n$ as
\begin{gather*}
f^n=(P_n,Q_n)\in(\C[z,w])^2, 
\end{gather*}
we have $\deg P_n =\deg_z P_n =d^n >\deg Q_n$ 
(\cite[Proposition 5.11]{DS}), 
and then
\begin{gather}
0< g^+=d^{-n}\log |P_n| + O(d^{-n})
\quad\text{and}\quad  Q_n = o(P_n)
\quad\text{as }n\to\infty  \label{estimate_on_B} 
\end{gather}
on $B^+\cap\bC^2$ locally uniformly, recalling also that
$\lim_{n\to\infty}f^n=[0:1:0]$ on $B^+$ locally uniformly. 

Fix a $2\times 2$ matrix
$A=\begin{pmatrix}
a_1 & a_2\\
a_3 & a_4
\end{pmatrix}\in \mathrm{M}(2,\C)$
satisfying the condition 
\begin{gather}
a_4\neq 0,\label{eq:nonvanish}
\end{gather}
so that for every $n\in\bN$, 
\begin{align}
\notag \det(D(f^n)- A)
=&J_{f^n} -a_1 \partial_w Q_n - a_4 	\partial_z P_n + a_3 \partial_w P_n  + a_2 	\partial_z Q_n +\det A\\
=&-a_1 \partial_w Q_n - a_4 	\partial_z P_n + a_3 \partial_w P_n  + a_2 	\partial_z Q_n +J_f^n+\det A
\in\bC[z,w]\label{eq:determinant}
\end{align}
is indeed of degree $d^n-1$. 

\begin{lemma}
 For each $j\in\{z,w\}$,
 \begin{gather}
  \partial_j P_n=2d^nP_n\partial_j g^+ +O(1)
  \quad\text{and}\quad \partial_j Q_n = o(d^nP_n)
  %\partial_j Q_n = {\red o(1)\cdot d^n P_n\partial_jg^+ +o(P_n)}
  \quad\text{as }n\to\infty
  \label{eq:partialQ}    
 \end{gather}
 on $B^+\cap\bC^2$ locally uniformly. 
\end{lemma}

\begin{proof}
 Pick any open concentric bidisks $D\Subset D'\Subset B^+\cap\bC^2$, 
 and fix $j\in\{z,w\}$. Let us write $D,D'$ as
 $D_1\times D_2,D_1'\times D_2'$, respectively. 
 
 By the former half in \eqref{estimate_on_B},
 we have $\inf_{D'}|P_n|>0$ if $n\gg 1$. We claim that
 \begin{gather}
  \partial_j g^+ 
  = d^{-n} \partial_j \log |P_n|+ O(d^{-n}) 
  =  \frac{1}{d^n}  \frac{\partial_j P_n}{2P_n}+ O(d^{-n})
  \quad\text{as }n\to\infty\label{eq:poisson}
 \end{gather}
 on $\overline{D}$ uniformly; indeed, 
 for every $z\in\overline{D_1}$, using Poisson's integral 
 of the function $w\mapsto g^+(z,w)-d^{-n}\log |P_n(z,w)|$ on $\partial D_2'$,
 the former half in \eqref{estimate_on_B} yields
 the asymptotic estimate \eqref{eq:poisson} on $\{z\}\times\overline{D_2}$ uniformly,
 and moreover, the implicit constant in $O$ depends only on $D$. 
 Hence the claim holds. 
 In particular, the former half in \eqref{eq:partialQ} holds.
 
 Similarly, using the latter half in \eqref{estimate_on_B} twice and 
 Cauchy's integral of the function $Q_n/P_n$ on 
 $\partial {D_1'}\times\partial {D_2'}$, 
 % where $D$ is written as $D_1\times D_2$ for some $D_1,D_2\subset\bC$, 
 we also have
 \begin{gather*}
  \frac{\partial_j Q_n}{P_n}=\frac{Q_n \partial_j P_n}{P_n^2}
  +\partial_j \left(\frac{Q_n}{P_n}\right) = 
  o(1)\cdot\frac{\partial_j P_n}{P_n}+o(1)\quad\text{as }n\to\infty
 \end{gather*}
 on $\overline{D}$ uniformly, which together with 
 \eqref{eq:poisson} and $\sup_D|\partial_j g^+|<\infty$
 % and the former half of \eqref{estimate_on_B} 
 yields
 \begin{gather*}
  \frac{\partial_j Q_n}{P_n} =o(d^n)+o(1)=o(d^n)\quad\text{as }n\to\infty
 \end{gather*}
on $\overline{D}$ uniformly. Hence the latter half in \eqref{eq:partialQ} also holds.
\end{proof}

By the pluriharmonicity of $g^+$ on $B^+$, the function
$a_4\partial_z g^+-a_3 \partial_w g^+$ is holomorphic on $B^+\cap\bC^2$. Set
\begin{gather*}
Y:=\bigl\{(z,w)\in B^+\cap\bC^2:\bigl(a_4\partial_z g^+ - a_3 \partial_w g^+\bigr)(z,w) =0\bigr\}. 
\end{gather*}
Recall the assumption that $a_4\neq 0$.

\begin{lemma}\label{th:nonvanish}
 $Y$ is an analytic hypersurface in $B^+\cap\bC^2$, 
 no irreducible component
 of which is horizontal, i.e., $\{w=w_0\}$ for some $w_0\in\bC$.
\end{lemma}

\begin{proof}
	Let us first show that $Y$ is not equal to $B^+ \cap \bC^2$. 
	Suppose to the contrary that $a_4\partial_z g^+- a_3 \partial_w g^+\equiv 0$ 
	on $B^+\cap\bC^2$.
	Then letting $L$ be the complex affine line $w=-(a_3/a_4)z$ in $\C^2$,
	there is $c\in\bR$ such that $g^+\equiv c$ on $L\cap B^+$. On the other hand,
	since the projective line $\overline{L}$ in $\bP^2$
	intersects $L_\infty$ at $[0:1:-a_3/a_4]\in L_\infty\setminus I^+$, 
	near which $g^+(z,w)-\log\max\{1,|z|\}$ extends pluriharmonically,
	we must have $c=g^+(z,w)=\log\max\{1,|z|\}+O(1)\to\infty$ as 
	$L\cap B^+\ni (z,w)\to[0:1:-a_3/a_4]$. This is a contradiction.
	Hence the former assertion holds.
	
	The latter assertion is shown similarly
	noting that the closure of any horizontal line intersects $L_\infty$
	at $[0:1:0]\in L_\infty\setminus  I^+$.
\end{proof}

Recall the computation \eqref{eq:determinant} 
of the polynomial $\det(D(f^n)- A)\in\bC[z,w]$ of degree $d^n-1$.
For every $n\in\bN$, set
\begin{gather*}
\phi_n=\phi_n[A]:=\frac{\log|\det(D(f^n)- A)|}{d^n-1},
\end{gather*}
which is a plurisubharmonic function on $\bC^2$ and satisfies
$\rd\rd^c\phi_n=[\det(D(f^n)- A)]/(d^n-1)$ as 
currents on $\bC^2$ by the Poincar\'e-Lelong formula.

\begin{lemma}\label{th:unifconv}
	We have $\phi_n=g^+ +O(nd^{-n})$ as $n\to\infty$
	on $B^+\cap(\bC^2\setminus Y)$ locally uniformly.
	Moreover, the family $(\phi_n)_n$ 
	is locally uniformly bounded from above on $\bC^2$.
\end{lemma}

\begin{proof}
 First, pick any open bidisk $D\Subset B^+\cap(\bC^2\setminus Y)$. Then
 by \eqref{eq:partialQ} and the former half in \eqref{estimate_on_B}, we have
 \begin{gather*}
  a_1 \partial_w Q_n + a_4\partial_z P_n - a_3 \partial_w P_n  - a_2\partial_z Q_n 
  = 2d^n P_n\cdot\bigl(a_4\partial_z g^+ -a_3 \partial_w g^+ + o(1)\bigr)
  % +o(d^n P_n) \\
  % = &2d^n P_n \bigl(a_4\partial_z g^+ -a_3 \partial_w g^+ + o(1)\bigr)
  \quad\text{as }n\to\infty 
 \end{gather*}
 on $\overline{D}$ uniformly, and 
 then using the former half in \eqref{estimate_on_B} again
 and $D\Subset B^+\cap(\bC^2\setminus Y)$, we have
\begin{align*} 
  \phi_n =&\frac{1}{d^n-1}\biggl(\log |P_n|
  +\log \left|2d^n \bigl(a_4\partial_z g^+ -a_3\partial_w g^+ + o(1)\bigr)-\frac{J_{f}^n +\det A}{P_n}\right|\biggr)\\
 =&\frac{1}{d^n-1} \log |P_n| +O(nd^{-n})=g^+ +O(nd^{-n})\quad\text{as }n\to\infty   
\end{align*}
 on $\overline{D}$ uniformly. Hence the former assertion holds. 
 
 Fix $(z_0,w_0)\in\bC^2$. By $L_\infty\setminus  I^+\subset B_+$ and
 the latter half in Lemma \ref{th:nonvanish}, 
 we have $\{|z-z_0|=r\}\times\{|w-w_0|=\epsilon\}\subset B^+\cap(\bC^2\setminus Y)$
 for $r\gg 1$ and $0<\epsilon\ll 1$, 
 so that by the former assertion and the maximum principle
 for the plurisubharmonic function $\phi_n$ on $\bC^2$, we have
 \begin{gather*}
  \sup_{\{|z-z_0|\le r\}\times\{|w-w_0|\le\epsilon\}}\phi_n
  \le\biggl(\sup_{\{|z-z_0|=r\}\times\{|w-w_0|=\epsilon\}}g^+\biggr)
  +O(nd^{-n})\quad\text{as }n\to\infty.
 \end{gather*}
 Hence the latter assertion also holds.
\end{proof}

Let us see
\begin{gather}
\lim_{n\to \infty}\frac{[\det(D(f^n)-A)]}{d^n-1}=T^+\quad\text{on }\bP^2
\tag{\ref{eq:equidistforward}$'$}\label{eq:normalized}
\end{gather}
as currents. First, 
let $\tilde{S}=\lim_{j\to\infty}[\det(D(f^{n_j})- A)]/(d^{n_j}-1)$
be any limit point, which is also a positive closed $(1,1)$-current on $\bP^2$ 
of mass $1$, of the sequence $([\det(D(f^n)- A)]/(d^n-1))_{n}$
of positive closed $(1,1)$-currents on $\bP^2$ of masses $1$. 
On the other hand, by Lemma \ref{th:unifconv} and
the {\em compactness principle} for
plurisubharmonic functions on a domain in $\bC^N$,
taking a subsequence of $(n_j)$ if necessary,
there is a plurisubharmonic function $\phi$ on $\bC^2$ such that
$\phi=\lim_{j\to\infty}\phi_{n_j}$ in $L^1_{\mathrm{loc}}(\bC^2,m_{4})$,
where $m_4$ is the Lebesgue measure on $\bC^2$.
Then we have $\tilde{S}|\bC^2=\rd\rd^c\phi$ on $\bC^2$ and,
by the former half in Lemma \ref{th:unifconv},
the plurisubharmonicity of $\phi$ on $\bC^2$, 
and the pluriharmonicity of $g^+$ on $B^+$,
we also have $\phi\equiv g^+$ on $B^+\cap\bC^2$.
Hence $\supp(\tilde{S}|\bC^2)\subset K^+$. 
Next, let $S$ be the {\em trivial extension} of $\rd\rd^c\phi$ to $\bP^2$ across $L_\infty$. 
%Since $S$ has mass $\leq 1$,  
It is a positive closed $(1,1)$-current on $\bP^2$
(cf.\ \cite[Theorem 2.7]{DS}) and supported by $\overline{K^+}= K^+\cup  I^+$.
Then by the uniqueness of $T^+$ mentioned above among such currents, 
there is $c\ge 0$ such that $S=c\cdot T^{+}$ on $\bP^2$.
Moreover, for the current of integration $[L]$
along any projective line $L\subset\bP^2\setminus  I^+$
other than $L_\infty$ and passing through $I^-$, if $R\gg 1$,
then we have $\phi\equiv g^+$ on 
$\{(z,w)\in\bC^2:\|(z,w)\|>R-1\}\cap L\subset B^+$, and in turn, 
recalling the definition of $S,T^+$ and using Stokes's formula,
we have
\begin{gather*}
c-1=\int_{\P^2} (S -T^+) \wedge [L] 
= \int_{\{\|(z,w)\|\le R\}} \rd\rd^c(\phi-g^+)\wedge [L] 
= \int_{\{\|(z,w)\|\le R\}\cap L} \rd\rd^c(\phi-g^+)
=0  
\end{gather*} 
(cf.\ \cite[Proof of Lemma 6.3]{DS}).
Hence $S=T^+$ on $\bP^2$. 
Consequently, $S|\bC^2=T^+|\bC^2=\rd\rd^c\phi=\tilde{S}|\bC^2$ on $\bC^2$, 
and then $\tilde{S}\ge S$ on $\bP^2$ by their construction.
Since both $\tilde{S},S$ are of masses $1$, 
we conclude that $\tilde{S}=S= T^+$ on $\bP^2$. 
Hence \eqref{eq:normalized} holds.

\begin{proof}[Proof of Theorem \ref{th:eigenvalue_automorphism}]
 Let $f$ be a H\'enon-type polynomial automorphism of $\C^2$ of degree $d>1$.
 Fix $\lambda\in\bC^*$, and set $A=\lambda I_2\in \mathrm{M}(2,\bC)$. Then
 using the chain rule and the equivariance of $T^+$ under affine coordinate changes
 on $\bC^2$, 
 we can assume that $f$ satisfies the normalization \eqref{eq:indeterminancies}, without loss of generality. Noting also that $A=\lambda I_2$ 
 satisfies the condition \eqref{eq:nonvanish}, the desired \eqref{eq:equidistforward}
 as currents on $\bP^2$ is nothing but \eqref{eq:normalized} 
 as currents on $\bP^2$ for this $A=\lambda I_2$.
\end{proof}

\begin{acknowledgement}
The first author was partially supported by JSPS Grant-in-Aid 
for Scientific Research (C), 15K04924. The authors were partially 
supported by Invitational Fellowships for Research in Japan 
(Short-term S18024) JSPS BRIDGE Fellowship 2018, and
thank for the hospitality of Universit\'e de Picardie Jules Verne and
Kyoto Institute of Technology, where they visited each other 
in 2018 and this work grew up.
\end{acknowledgement}

%\bibliography{biblio}

\begin{thebibliography}{10}

\bibitem{baker-hsia}
Matthew~H. Baker and Liang-Chung Hsia.
\newblock Canonical heights, transfinite diameters, and polynomial dynamics.
\newblock {\em J. Reine Angew. Math.}, 585:61--92, 2005.

\bibitem{BR06}
Matthew~H. Baker and Robert Rumely.
\newblock Equidistribution of small points, rational dynamics, and potential
  theory.
\newblock {\em Ann. Inst. Fourier (Grenoble)}, 56(3):625--688, 2006.

\bibitem{BR10}
Matthew~H. Baker and Robert Rumely.
\newblock {\em Potential theory and dynamics on the {B}erkovich projective
  line}, volume 159 of {\em Mathematical Surveys and Monographs}.
\newblock American Mathematical Society, Providence, RI, 2010.

\bibitem{Bed_Smi_critical}
Eric Bedford and John Smillie.
\newblock Polynomial diffeomorphisms of {${\bf C}^2$}. {V}. {C}ritical points
  and {L}yapunov exponents.
\newblock {\em J. Geom. Anal.}, 8(3):349--383, 1998.

\bibitem{Berkovichbook}
Vladimir~G. Berkovich.
\newblock {\em Spectral theory and analytic geometry over non-{A}rchimedean
  fields}, volume~33 of {\em Mathematical Surveys and Monographs}.
\newblock American Mathematical Society, Providence, RI, 1990.

\bibitem{rudiments}
Fran{\c{c}}ois Berteloot and Volker Mayer.
\newblock {\em Rudiments de dynamique holomorphe}, volume~7 of {\em Cours
  Sp\'ecialis\'es}.
\newblock Soci\'et\'e Math\'ematique de France, Paris, 2001.

\bibitem{brolin}
Hans Brolin.
\newblock Invariant sets under iteration of rational functions.
\newblock {\em Ark. Mat.}, 6:103--144 (1965), 1965.

\bibitem{ACL}
Antoine Chambert-Loir.
\newblock Mesures et \'equidistribution sur les espaces de {B}erkovich.
\newblock {\em J. Reine Angew. Math.}, 595:215--235, 2006.

\bibitem{dinhsibony2}
Tien-Cuong Dinh and Nessim Sibony.
\newblock Dynamics in several complex variables: endomorphisms of projective
  spaces and polynomial-like mappings.
\newblock In {\em Holomorphic dynamical systems}, volume 1998 of {\em Lecture
  Notes in Math.}, pages 165--294. Springer, Berlin, 2010.

\bibitem{DS}
Tien-Cuong Dinh and Nessim Sibony.
\newblock Rigidity of {J}ulia sets for {H}\'{e}non type maps.
\newblock {\em J. Mod. Dyn.}, 8(3-4):499--548, 2014.

\bibitem{FRL}
Charles Favre and Juan Rivera-Letelier.
\newblock \'{E}quidistribution quantitative des points de petite hauteur sur la
  droite projective.
\newblock {\em Math. Ann.}, 335(2):311--361, 2006.

\bibitem{Forn_Sib}
John~Erik Forn\ae~ss and Nessim Sibony.
\newblock Complex dynamics in higher dimensions.
\newblock In {\em Complex potential theory ({M}ontreal, {PQ}, 1993)}, volume
  439 of {\em NATO Adv. Sci. Inst. Ser. C Math. Phys. Sci.}, pages 131--186.
  Kluwer Acad. Publ., Dordrecht, 1994.
\newblock Notes partially written by Estela A. Gavosto.

\bibitem{FvdP04}
Jean Fresnel and Marius van~der Put.
\newblock {\em Rigid analytic geometry and its applications}, volume 218 of
  {\em Progress in Mathematics}.
\newblock Birkh\"{a}user Boston, Inc., Boston, MA, 2004.

\bibitem{Gauthier_Vigny_derivative}
Thomas Gauthier and Gabriel Vigny.
\newblock Distribution of points with prescribed derivative in polynomial
  dynamics.
\newblock {\em Riv. Math. Univ. Parma (N.S.)}, 8(2):247--270, 2017.

\bibitem{Hormander83}
Lars H{\"o}rmander.
\newblock {\em The analysis of linear partial differential operators. {I}},
  volume 256 of {\em Grundlehren der Mathematischen Wissenschaften [Fundamental
  Principles of Mathematical Sciences]}.
\newblock Springer-Verlag, Berlin, 1983.
\newblock Distribution theory and Fourier analysis.

\bibitem{Jonsson15}
Mattias Jonsson.
\newblock Dynamics on berkovich spaces in low dimensions.
\newblock In {\em Berkovich Spaces and Applications}, pages 205--366. Springer,
  2015.

\bibitem{Milnor4}
John Milnor.
\newblock {\em Dynamics in one complex variable}, volume 160 of {\em Annals of
  Mathematics Studies}.
\newblock Princeton University Press, Princeton, NJ, third edition, 2006.

\bibitem{OkuDivisor}
Y{\^u}suke Okuyama.
\newblock Effective divisors on the projective line having small diagonals and
  small heights and their application to adelic dynamics.
\newblock {\em Pacific J. Math.}, 280(1):141--175, 2016.

\bibitem{okuadelically}
Y\^{u}suke Okuyama.
\newblock Adelically summable normalized weights and adelic equidistribution of
  effective divisors having small diagonals and small heights on the
  {B}erkovich projective lines.
\newblock In {\em Algebraic number theory and related topics 2014}, RIMS
  K\^{o}ky\^{u}roku Bessatsu, B64, pages 55--66. Res. Inst. Math. Sci. (RIMS),
  Kyoto, 2017.

\bibitem{OkuyamaDerivatives}
Y\^{u}suke Okuyama.
\newblock Value distribution of the sequences of the derivatives of iterated
  polynomials.
\newblock {\em Ann. Acad. Sci. Fenn. Math.}, 42(2):563--574, 2017.

\bibitem{Okuapripri}
Y{\^u}suke {Okuyama}.
\newblock {An a priori bound of rational functions on the Berkovich projective
  line}.
\newblock {\em arXiv e-prints}, page arXiv:1805.07668, May 2018.

\bibitem{ransford}
Thomas Ransford.
\newblock {\em Potential theory in the complex plane}, volume~28 of {\em London
  Mathematical Society Student Texts}.
\newblock Cambridge University Press, Cambridge, 1995.

\bibitem{Juan03}
Juan Rivera-Letelier.
\newblock Dynamique des fonctions rationnelles sur des corps locaux.
\newblock {\em Ast\'erisque}, (287):xv, 147--230, 2003.
\newblock Geometric methods in dynamics. II.

\bibitem{RS97}
Alexander Russakovskii and Bernard Shiffman.
\newblock Value distribution for sequences of rational mappings and complex
  dynamics.
\newblock {\em Indiana Univ. Math. J.}, 46(3):897--932, 1997.

\bibitem{Yamanoi13}
Katsutoshi Yamanoi.
\newblock Zeros of higher derivatives of meromorphic functions in the complex
  plane.
\newblock {\em Proc. Lond. Math. Soc. (3)}, 106(4):703--780, 2013.

\bibitem{Ye11}
Hexi Ye.
\newblock The {S}chwarzian derivative and polynomial iteration.
\newblock {\em Conform. Geom. Dyn.}, 15:113--132, 2011.

\end{thebibliography}
%\bibliographystyle{plain}
\end{document}